\documentclass[a4paper,10pt]{article}\usepackage{palatino,amsmath,amsthm,amssymb,bbm,vmargin,arydshln,mathpazo}\title{\textbf{First exit time from a bounded interval for pseudo-processes driven by the equation $\partial/\partial t=(-1)^{N-1} \partial^{2N}\!/\partial x^{2N}$}} \author{Aim\'e LACHAL\footnote{Affiliation: Universit\'e de Lyon/Institut Camille Jordan CNRS UMR5208 \hspace{.3\linewidth}\mbox{} Postal address: \textsl{Institut National des Sciences Appliqu\'ees de Lyon} \hspace{.4\linewidth}\mbox{} P\^ole de Math\'ematiques, B\^atiment L\'eonard de Vinci \hspace{.5\linewidth}\mbox{} 20 avenue Albert Einstein, 69621 Villeurbanne Cedex, France \hspace{.4\linewidth}\mbox{} \mbox{E-mail:~\texttt{aime.lachal@insa-lyon.fr}} \hspace{.6\linewidth}\mbox{} URL:~\texttt{http://maths.insa-lyon.fr/\~~\!\!\!lachal}}}\date{}\setmarginsrb{30mm}{15mm}{30mm}{15mm}{0mm}{10mm}{0mm}{10mm}\newtheorem{theorem}{Theorem}\newtheorem{corollary}{Corollary}\newtheorem{heuristic}{Heuristic}\newtheorem{remark}{Remark}\numberwithin{equation}{section}

\newcommand{\ind}{1\hspace{-.27em}\mbox{\rm l}}
\newcommand{\lqn}[1]{\noalign{\noindent $\displaystyle{#1}$}}

\begin{document}

\maketitle

\begin{abstract} Let $N$ be an integer greater than $1$. We consider the pseudo-process $X=(X_t)_{t \ge 0}$ driven by the high-order heat-type equation $\partial/\partial t=(-1)^{N-1} \partial^{2N}\!/\partial x^{2N}$. Let us introduce the first exit time $\tau_{ab}$ from a bounded interval $(a,b)$ by $X$ ($a,b\in \mathbb{R}$) together with the related location, namely $X_{\tau_{ab}}$.\\ \indent In this paper, we provide a representation of the joint pseudo-distribution of the vector $(\tau_{ab},X_{\tau_{ab}})$ by means of some determinants. The method we use is based on a Feynman-Kac-like functional related to the pseudo-process $X$ which leads to a boundary value problem. In particular, the pseudo-distribution of $X_{\tau_{ab}}$ admits a fine expression involving famous Hermite interpolating polynomials.\end{abstract}

\begin{small} \noindent\textit{MSC:} primary 60G20; secondary 60G40; 60K99\\[1ex] \noindent \textit{Keywords:} Pseudo-Brownian motion; First exit time; Laplace transform; Hermite interpolating polynomials \end{small}

\section{Introduction}

Let $N$ be an integer greater than $1$ and set $\kappa_{_{\!N}}=(-1)^{N-1}$.
We consider the pseudo-process $(X_t)_{t\ge 0}$ driven by the high-order
heat-type equation $\partial/\partial t=\kappa_{_{\!N}} \partial^{2N}\!/\partial x^{2N}$,
the so-called pseudo-Brownian motion. This is the pseudo-Markov process
with independent and stationary increments, associated with the \textit{signed}
heat-type kernel $p(t;x)$ which is the elementary solution of the foregoing equation.
The kernel $p(t;x)$ is characterized by its Fourier transform:
\[
\int_{-\infty}^{+\infty} \mathrm{e}^{\mathrm{i} ux}\,p(t;x)\,\mathrm{d} x
=\mathrm{e}^{-tu^{2N}}\!.
\]
We define the related transition kernel as $p(t;x,y)=p(t;x-y)$
for any time $t>0$ and any real numbers $x,y$, which represents
the pseudo-probability that the pseudo-process started at $x$
is in state $y$ at time $t$. In symbols,
\[
\mathbb{P}_{\!x}\{X_t\in \mathrm{d}y\}=p(t;x,y)\,\mathrm{d}y.
\]
The $\mathbb{P}_{\!x}$, $x\in\mathbb{R}$, define a family of \textit{signed}
measures whose total mass equals one:
\[
\mathbb{P}_{\!x}\{X_t\in\mathbb{R}\}=\int_{-\infty}^{+\infty}
p(t;x,y)\,\mathrm{d}y=1.
\]
The transition kernel $p(t;x,y)$ satisfies the backward and forward Kolmogorov
equations
\[
\frac{\partial p}{\partial t}(t;x,y)=\kappa_{_{\!N}}
\frac{\partial^{2N}\!p}{\partial x^{2N}}(t;x,y)
=\kappa_{_{\!N}} \frac{\partial^{2N}\!p}{\partial y^{2N}}(t;x,y).
\]

To be more precise, let us recall that the pseudo-Markov process
$(X_t)_{t\ge 0}$ is defined according to the usual chain rule:
for any positive integer $n$, for any times $t_1,\dots,t_n$ such that $0<t_1<\dots<t_n$
and any real numbers $x_1,\dots,x_n$, and setting $t_0=0$, $x_0=x$,
\begin{equation}\label{chain-rule}
\mathbb{P}_{\!x}\{X_{t_1}\in \mathrm{d}x_1,\dots,X_{t_n}\in \mathrm{d}x_n\}
=\left(\prod_{k=1}^n p(t_k-t_{k-1};x_{k-1},x_k)\right)\mathrm{d}x_1\dots\mathrm{d}x_n.
\end{equation}
In particular, by setting $T_t\phi(x)=\mathbb{E}_x[\phi(X_t)]$ for any time $t$,
any real number $x$ and any bounded $\mathcal{C}^{2N}$-function $\phi$,
the family $(T_t)_{t\ge 0}$ is a semi-group of operators whose infinitesimal
generator $\mathcal{G}$ is given by
\begin{equation}\label{infinitesimal-generator2}
\mathcal{G} \phi(x)=\lim_{h\to 0^+} \frac{1}{h}\left[\,\mathbb{E}_x[\phi(X_h)]-\phi(x)\right]
=\kappa_{_{\!N}} \,\phi^{(2N)}(x).
\end{equation}
Above and throughout the paper, for any non-negative integer $\ell$, $\phi^{(\ell)}$
stands for the derivative of $\phi$ of order $\ell$.

The very notion of pseudo-process in a general framework goes back to
Daletskii and Fomin in 1965 (\cite{df}).
The reader can find an extensive literature on the particular case of
pseudo-Brownian motion.
For instance, let us quote the works of Beghin, Cammarota, Hochberg, Krylov,
Lachal, Nakajima, Nikitin, Nishioka, Orsingher, Ragozina (\cite{bho} to \cite{cl2},
\cite{hoch,ho}, \cite{kry} to \cite{ors})
and the references therein. These papers deal with several functionals
related to pseudo-Brownian motion: sojourn time in a bound\-ed or not interval,
first overshooting time of a \textit{single} level, maximum or minimum up to a fixed time...
Let us mention also other interesting works : one dealing with high-order
Schr\"odinger-type equation $\partial/\partial t=\mathrm{i}\,\partial^{2N}\!/\partial x^{2N}$
which is related to the so-called Feynman-Kac measure \cite{asf},
as well as \cite{isf} in which the authors develop an alternative
and more probabilistic approach to pseudo-processes.

In~\cite{la2,la3}, we obtained the pseudo-distribution of the first overshooting
time of a single threshold, together with the corresponding location at this time.
In symbols, if $\tau_a$ denotes the first overshooting time of a fixed
level $a$ (upwards or downwards), we derived the joint pseudo-distribution of
the couple $(\tau_a,X_{\tau_a})$. Therein, we used an extension of famous
Spitzer's identity.
In~\cite{la2,la3} and, in the particular case $N=2$, in~\cite{nish1,nish2},
the authors observed a curious fact concerning the pseudo-distribution
of $X_{\tau_a}$: it is a linear combination of the Dirac distribution and its
successive derivatives (in the sense of Schwartz distributions):
\begin{equation}\label{dirac1}
\mathbb{P}_{\!x}\{X_{\tau_a}\in \mathrm{d} z\}/\mathrm{d} z
=\sum_{k=0}^{N-1}\frac{(a-x)^k}{k!}\,\delta_a^{(k)}(z).
\end{equation}
The quantity $\delta_a^{(k)}$ is to be understood as the functional acting
on test functions $\phi$ according as $\langle \delta_a^{(k)},\phi \rangle
=(-1)^k\phi^{(k)}(a)$. Formula~(\ref{dirac1}) says that the overshoot
through level $a$ should be actually concentrated at $a$.
The appearance of the Schwartz-Dirac distribution $\delta_a$
together with its successive derivatives can be interpreted by means
of ``multipoles'' in reference to electric dipoles as in~\cite{la2,la3} and,
for $N=2$, in~\cite{nish1,nish2}. In particular, therein, $\delta_a$ and
$\delta_a'$ are respectively named ``monopole'' and ``dipole''.
We refer the reader to~\cite{nish3} for a detailed account on monopoles and dipoles.
An explanation of this curious fact should be found in
considering a linear pseudo-random walk with $2N$ consecutive neighbours around
each sites. Indeed, after suitably normalizing such a walk, the neighbours
cluster into a single site and form a multipole; see the draft~\cite{la6}.

Till now, the first exit time from a bounded interval, or, equivalently,
the first overshooting time of a \textit{double} threshold has not yet
been considered. This is the purpose of this work.

Let us introduce the first exit time from $(a,b)$
($a,b$ being real numbers such that $a<b$) for $(X_t)_{t\ge 0}$:
\[
\tau_{ab}=\inf\{t\ge 0: X_t\notin(a,b)\}
\]
with the usual convention that $\inf\varnothing=+\infty$.
In this paper, we tackle the problem of finding the pseudo-distribution
related to the double threshold: we provide a representation for the joint pseudo-distribution
of the couple $(\tau_{ab},X_{\tau_{ab}})$.
This representation involves some determinants;
this is the object of Theorems~\ref{theo-Phi} and \ref{theo-joint-dist-Xtau-ab}.
For the location $X_{\tau_{ab}}$, we have the following counterpart
to~(\ref{dirac1}); see Theorem~\ref{theo-dist-Xab}:
\begin{equation}\label{dirac2}
\mathbb{P}_{\!x}\{X_{\tau_{ab}}\in \mathrm{d} z\}/\mathrm{d} z
=\sum_{k=0}^{N-1} (-1)^kH_k^-(x)\,\delta_a^{(k)}(z)
+\sum_{k=0}^{N-1} (-1)^kH_k^+(x)\,\delta_b^{(k)}(z)
\end{equation}
where the functions $H_k^{\pm}$, $0\le k\le N-1$, are the classical
Hermite interpolating polynomials of degree $(2N-1)$ related to points $a$ and
$b$ satisfying
\[
(H_k^-)^{(\ell)}(a)=(H_k^+)^{(\ell)}(b)=\delta_{k\ell},\quad
(H_k^+)^{(\ell)}(a)=(H_k^-)^{(\ell)}(b)=0, \quad 0\le \ell\le N-1.
\]
Above, the quantity $\delta_{k\ell}$ denotes the usual Kronecker symbol:
if $k=\ell$, $\delta_{k\ell}=1$, else $\delta_{k\ell}=0$.
They explicitly write as
\begin{align*}
H_k^-(x) &=\left(\frac{b-x}{b-a}\right)^{\!\!N}\frac{(x-a)^k}{k\,!}
\sum_{\ell=0}^{N-k-1}\binom{\ell+N-1}{\ell} \!\left(\frac{x-a}{b-a}\right)^{\!\!\ell},
\\
H_k^+(x) &=\left(\frac{x-a}{b-a}\right)^{\!\!N} \frac{(x-b)^k}{k\,!}
\sum_{\ell=0}^{N-k-1}\binom{\ell+N-1}{\ell} \!\left(\frac{b-x}{b-a}\right)^{\!\!\ell}.
\end{align*}
In particular, we can deduce from~(\ref{dirac2}) the ``ruin pseudo-probabilities'',
that is, the pseudo-probabilities of overshooting one level ($a$ or $b$) before
the other one; see Corollary~\ref{cor-ruin}.

These results have been announced without any proof in a survey on
pseudo-Brownian motion, \cite{la5}, after a conference held in Madrid (IWAP 2010).

Throughout the paper, the function $\varphi$ denotes any $(N-1)$ times
differentiable function.

\section{Feynman-Kac functional}

We start from the following fact: in~\cite{la2,la3}, we first obtained the
pseudo-distribution of the couple $(\sup_{0\le s\le t}X_s,X_t)$ by making use
of an extension of Spitzer's identity. From this, we deduced that of the couple
$(\tau_a,X_{\tau_a})$ and we made the observation that, for any $\lambda\ge 0$
and any $(N-1)$ times differentiable bounded function~$\varphi$,
the Feynman-Kac functional
$\Phi(x)=\mathbb{E}_x\!\left(\mathrm{e}^{-\lambda\tau_a}\varphi(X_{\tau_a})
\,\ind_{\{\tau_a<+\infty\}}\right)$ solves the boundary value problem
\begin{equation}\label{BVP-old}
\left\{\!\!\begin{array}{l}
\kappa_{_{\!N}}\Phi^{(2N)}(x)=\lambda \,\Phi(x),\quad x\in(-\infty,a)\text{ (or $x\in(a,+\infty)$)},
\\[1ex]
\Phi^{(k)}(a)=\varphi^{(k)}(a)\quad \text{for } k\in\{0,1,\dots,N-1\}.
\end{array}\right.
\end{equation}
So, we state the heuristic that an analogous boundary value problem
should hold for the Feynman-Kac functional related to $\tau_{ab}$.
The results obtained here through this approach coincide with
limiting results deduced from a suitable pseudo-random walk
studied in~\cite{la6}. Moreover, when taking the limit as $a$ goes to $-\infty$
or $b$ goes to $+\infty$ in the present results, we retrieve the pseudo-distribution
of $(\tau_a,X_{\tau_a})$ obtained in~\cite{la3}.
So, these observations comfort us in our heuristic.
Actually, our purpose in this work is essentially concentrated in calculating
the pseudo-distri\-bution of $(\tau_{ab},X_{\tau_{ab}})$.

As pointed out in several works on pseudo-processes, pseudo-Brownian motion
is properly defined only on the set of dyadic times and ad-hoc definitions
should be taken for computing certain functionals of this pseudo-process
depending on a continuous set of times; see, e.g., \cite{la3} and, in the particular
case $N=2$, \cite{nish2}. Roughly speaking, the dense subset of dyadic times
is appropriate because of the usual property that for any $n\in\mathbb{N}$,
$\{k/2^n,k\in\mathbb{N}\} \subset \{k/2^{n+1},k\in\mathbb{N}\}$.
Indeed, this latter permits to view the pseudo-process $(X_t)_{t\ge 0}$ as an informal limit
of the family of step-processes $(X_{n,t})_{t\ge 0}$ defined according to
the following sampling procedure:
$$
X_{n,t}=\sum_{k=0}^{\infty} \mathbbm{1}_{[k/2^n,(k+1)/2^n)}(t) X_{k/2^n}.
$$
For each fixed $n\in\mathbb{N}$, the sequence $(X_{k/2^n})_{k\in\mathbb{N}}$ can be
correctly defined thanks to~(\ref{chain-rule}). But the fact that
$\int_{-\infty}^{+\infty} |p(t;x)|\,\mathrm{d} x>1$
prevent us from applying the classical extension theorem of Kolmogorov
for finding \textit{a priori} a $\sigma$-additive measure on the usual space
of right-continuous functions on $[0,+\infty)$ which have left-hand limits,
measure whose finite projections would yield the finite-dimensional pseudo-distributions
of the sequence $(X_{k/2^n})_{k\in\mathbb{N}}$.

For our concern, we set
\[
\tau_{ab,n}=\frac{1}{2^n} \min\{k\in\mathbb{N}:X_{k/2^n}\notin (a,b)\}
\]
and, for $x\in(a,b)$,
\[
\Phi_n(x)=\mathbb{E}_x\!\left(\mathrm{e}^{-\lambda\tau_{ab,n}}\varphi(X_{\tau_{ab,n}})
\,\ind_{\{\tau_{ab,n}<+\infty\}}\right)\!.
\]
Then, we define the Feynman-Kac functional
$\Phi(x)=\mathbb{E}_x\!\left(\mathrm{e}^{-\lambda\tau_{ab}}\varphi(X_{\tau_{ab}})
\,\ind_{\{\tau_{ab}<+\infty\}}\right)$ as the limit
\[
\Phi(x)\overset{\text{def}}{=}\lim_{n\to +\infty}\Phi_n(x)
\]
and we state below the analogue to~(\ref{BVP-old}).
%
\begin{heuristic}\label{theo-FK}
For any $\lambda\ge 0$ and any $(N-1)$ times differentiable bounded
function~$\varphi$, the Feynman-Kac functional
$\Phi(x)=\mathbb{E}_x\!\left(\mathrm{e}^{-\lambda\tau_{ab}}\varphi(X_{\tau_{ab}})
\,\ind_{\{\tau_{ab}<+\infty\}}\right)$ solves the boundary value problem
\begin{equation}\label{BVP1}
\left\{\!\!\begin{array}{l}
\kappa_{_{\!N}}\Phi^{(2N)}(x)=\lambda \,\Phi(x),\quad x\in(a,b),
\\[1ex]
\Phi^{(k)}(a)=\varphi^{(k)}(a)\quad\text{and}\quad \Phi^{(k)}(b)
=\varphi^{(k)}(b) \quad \text{for } k\in\{0,1,\dots,N-1\}.
\end{array}\right.
\end{equation}
\end{heuristic}
\section{Joint pseudo-distribution of $\left(\tau_{ab},X_{\tau_{ab}}\right)$}

In this section, we solve boundary value problem~(\ref{BVP1}) in order
to derive the joint pseudo-probab\-ility of $\left(\tau_{ab},X_{\tau_{ab}}\right)$.
In this way, if we choose $\varphi(x)=\mathrm{e}^{\mathrm{i}\mu x}$, $\mu\in\mathbb{R}$,
we first obtain its Laplace-Fourier transform. Actually, the results
we derived hold true for any $(N-1)$ times differentiable function~$\varphi$.

Let us introduce the $(2N)$th roots of $\kappa_{_{\!N}}$:
$\theta_{\ell}=\mathrm{e}^{\mathrm{i}\frac{2\ell+N-1}{2N} \pi}$,
$1\le \ell\le 2N$. We have $\theta_{\ell}^{2N}=\kappa_{_{\!N}}$.
For any complex number $z$, we set $e_{\lambda}^z=\mathrm{e}^{\lambda^{1/(2N)}z}$.

%
\begin{theorem}\label{theo-Phi}
The Feynman-Kac functional related to $\left(\tau_{ab},X_{\tau_{ab}}\right)$ admits the following
representation:
\begin{equation}\label{exp-Phi}
\mathbb{E}_x\!\left(\mathrm{e}^{-\lambda\tau_{ab}}\varphi(X_{\tau_{ab}})
\,\ind_{\{\tau_{ab}<+\infty\}}\right)
=\sum_{k=0}^{N-1} \lambda^{-\frac{k}{2N}}\frac{\Delta_k^-(\lambda;x)}{\Delta(\lambda)}\,\varphi^{(k)}(a)
+\sum_{k=0}^{N-1} \lambda^{-\frac{k}{2N}}\frac{\Delta_k^+(\lambda;x)}{\Delta(\lambda)}\,\varphi^{(k)}(b)
\end{equation}
where the quantities $\Delta(\lambda)$ and $\Delta_k^{\pm}(\lambda;x)$ are the
determinants below:
\[
\Delta(\lambda)
=\begin{vmatrix}
\,e_{\lambda}^{\theta_1a}                 & \cdots & e_{\lambda}^{\theta_{2N}a}
\\[.5ex]
\,\theta_1\,e_{\lambda}^{\theta_1a}       & \cdots & \theta_{2N}\,e_{\lambda}^{\theta_{2N}a}
\\[.5ex]
\,\vdots                                  &        &\vdots
\\[.5ex]
\,\theta_1^{N-1}\,e_{\lambda}^{\theta_1a} & \cdots & \theta_{2N}^{N-1}\,e_{\lambda}^{\theta_{2N}a}
\\[-0.5ex]
\hdotsfor{3}
\\
\,e_{\lambda}^{\theta_1b}                 & \cdots & e_{\lambda}^{\theta_{2N}b}
\\[.5ex]
\,\theta_1\,e_{\lambda}^{\theta_1b}       & \cdots & \theta_{2N}\,e_{\lambda}^{\theta_{2N}b}
\\[.5ex]
\,\vdots                                  &        &\vdots
\\[.5ex]
\,\theta_1^{N-1}\,e_{\lambda}^{\theta_1b} & \cdots & \theta_{2N}^{N-1}\,e_{\lambda}^{\theta_{2N}b}
\end{vmatrix}
\]
and
\[
\Delta_k^-(\lambda;x)=\begin{vmatrix}
\,e_{\lambda}^{\theta_1a}                 & \cdots & e_{\lambda}^{\theta_{2N}a}
\\[.5ex]
\,\vdots                                  &        & \vdots
\\[.5ex]
\,\theta_1^{k-1}\,e_{\lambda}^{\theta_1a} & \cdots & \theta_{2N}^{k-1}\,e_{\lambda}^{\theta_{2N}a}
\\[.5ex]
\,e_{\lambda}^{\theta_1x}                 & \cdots & e_{\lambda}^{\theta_{2N}x}
\\[.5ex]
\,\theta_1^{k+1}\,e_{\lambda}^{\theta_1a} & \cdots & \theta_{2N}^{k+1}\,e_{\lambda}^{\theta_{2N}a}
\\[.5ex]
\,\vdots                                  &        & \vdots
\\[.5ex]
\,\theta_1^{N-1}\,e_{\lambda}^{\theta_1a} & \cdots & \theta_{2N}^{N-1}\,e_{\lambda}^{\theta_{2N}a}
\\[-0.5ex]
\hdotsfor{3}
\\
\,e_{\lambda}^{\theta_1b}                 & \cdots & e_{\lambda}^{\theta_{2N}b}
\\[.5ex]
\,\vdots                                  &        & \vdots
\\[.5ex]
\,\theta_1^{N-1}\,e_{\lambda}^{\theta_1b} & \cdots & \theta_{2N}^{N-1}\,e_{\lambda}^{\theta_{2N}b}
\end{vmatrix},
\quad
\Delta_k^+(\lambda;x)=\begin{vmatrix}
\,e_{\lambda}^{\theta_1a}                 & \cdots & e_{\lambda}^{\theta_{2N}a}
\\[.5ex]
\,\vdots                                  &        & \vdots
\\[.5ex]
\,\theta_1^{N-1}\,e_{\lambda}^{\theta_1a} & \cdots & \theta_{2N}^{N-1}\,e_{\lambda}^{\theta_{2N}a}
\\[-0.5ex]
\hdotsfor{3}
\\
\,e_{\lambda}^{\theta_1b}                 & \cdots & e_{\lambda}^{\theta_{2N}b}
\\[.5ex]
\,\vdots                                  &        & \vdots
\\[.5ex]
\,\theta_1^{k-1}\,e_{\lambda}^{\theta_1b} & \cdots & \theta_{2N}^{k-1}\,e_{\lambda}^{\theta_{2N}b}
\\[.5ex]
\,e_{\lambda}^{\theta_1x}                 & \cdots & e_{\lambda}^{\theta_{2N}x}
\\[.5ex]
\,\theta_1^{k+1}\,e_{\lambda}^{\theta_1b} & \cdots & \theta_{2N}^{k+1}\,e_{\lambda}^{\theta_{2N}b}
\\[.5ex]
\,\vdots                                  &        & \vdots
\\[.5ex]
\,\theta_1^{N-1}\,e_{\lambda}^{\theta_1b} & \cdots & \theta_{2N}^{N-1}\,e_{\lambda}^{\theta_{2N}b}
\end{vmatrix}.
\]
The functions $x\mapsto \Delta_k^{\pm}(\lambda;x)$, $0\le k\le N-1$,
are the solutions of the boundary value problems
\[
\left\{\!\!\begin{array}{l}
(\Delta_k^-)^{(2N)}(\lambda;x)
=\kappa_{_{\!N}}\lambda \,\Delta_k^-(\lambda;x),
\\[.5ex]
(\Delta_k^-)^{(\ell)}(\lambda;a)
=\delta_{k\ell} \,\lambda^{\ell/(2N)}\Delta(\lambda),\;
(\Delta_k^-)^{(\ell)}(\lambda;b)=0
\quad\text{for } \ell\in\{0,\dots,N-1\},
\end{array}\right.
\]
\[
\left\{\!\!\begin{array}{l}
(\Delta_k^+)^{(2N)}(\lambda;x)
=\kappa_{_{\!N}}\lambda \,\Delta_k^+(\lambda;x),
\\[.5ex]
(\Delta_k^+)^{(\ell)}(\lambda;a)=0,\;
(\Delta_k^+)^{(\ell)}(\lambda;b)=
\delta_{k\ell} \,\lambda^{\ell/(2N)}\Delta(\lambda)
\quad\text{for } \ell\in\{0,\dots,N-1\}.
\end{array}\right.
\]
\end{theorem}
%
\begin{proof}
The solution of linear boundary value problem~(\ref{BVP1}) has the form
$\Phi(x)=\sum_{\ell=1}^{2N} \alpha_{\ell}\,e_{\lambda}^{\theta_{\ell}x}$
where the coefficients $\alpha_{\ell}$, $1\le \ell \le 2N$, satisfy
the linear system below:
\begin{equation}\label{linear-system}
\left\{\begin{array}{ll}
\displaystyle\sum_{\ell=1}^{2N} \theta_{\ell}^k \,e_{\lambda}^{\theta_{\ell}a}\,\alpha_{\ell}
=\lambda^{-\frac{k}{2N}}\varphi^{(k)}(a),& 0\le k\le N-1,
\\[3ex]
\displaystyle\sum_{\ell=1}^{2N} \theta_{\ell}^k \,e_{\lambda}^{\theta_{\ell}b}\,\alpha_{\ell}
=\lambda^{-\frac{k}{2N}}\varphi^{(k)}(b),& 0\le k\le N-1.
\end{array}\right.
\end{equation}
This system can be solved by using Cramer's formulae:
\[
\alpha_{\ell}=\frac{\Delta_{\ell}(\lambda,\varphi)}{\Delta(\lambda)},
\quad 1\le\ell\le 2N,
\]
where $\Delta(\lambda)$ is the determinant displayed in Theorem~\ref{theo-Phi}
and $\Delta_{\ell}(\lambda,\varphi)$ is the determinant deduced from
$\Delta(\lambda)$ by replacing its $\ell$th column by the right-hand
side of~(\ref{linear-system}), that is
\[
\Delta_{\ell}(\lambda,\varphi)
=\begin{vmatrix}
\,e_{\lambda}^{\theta_1a}                 & \cdots & e_{\lambda}^{\theta_{\ell-1}a}                        & \varphi(a)                                  & e_{\lambda}^{\theta_{\ell+1}a}                        & \cdots & e_{\lambda}^{\theta_{2N}a}
\\[.5ex]
\,\theta_1\,e_{\lambda}^{\theta_1a}       & \cdots & \theta_{\ell-1}\,e_{\lambda}^{\theta_{\ell-1}a}       & \lambda^{-\frac{1}{2N}}\varphi'(a)          & \theta_{\ell+1}\,e_{\lambda}^{\theta_{\ell+1}a}       & \cdots & \theta_{2N}\,e_{\lambda}^{\theta_{2N}a}
\\[.5ex]
\,\vdots                                  &        & \vdots                                                & \vdots                                      & \vdots                                                &        & \vdots
\\[.5ex]
\,\theta_1^{N-1}\,e_{\lambda}^{\theta_1a} & \cdots & \theta_{\ell-1}^{N-1}\,e_{\lambda}^{\theta_{\ell-1}a} & \lambda^{-\frac{N-1}{2N}}\varphi^{(N-1)}(a) & \theta_{\ell+1}^{N-1}\,e_{\lambda}^{\theta_{\ell+1}a} & \cdots & \theta_{2N}^{N-1}\,e_{\lambda}^{\theta_{2N}a}
\\[-0.5ex]
\hdotsfor{7}
\\
\,e_{\lambda}^{\theta_1b}                 & \cdots & e_{\lambda}^{\theta_{\ell-1}b}                        & \varphi(b)                                  & e_{\lambda}^{\theta_{\ell+1}b}                        & \cdots & e_{\lambda}^{\theta_{2N}b}
\\[.5ex]
\,\theta_1\,e_{\lambda}^{\theta_1b}       & \cdots & \theta_{\ell-1}\,e_{\lambda}^{\theta_{\ell-1}b}       & \lambda^{-\frac{1}{2N}}\varphi'(b)          & \theta_{\ell+1}\,e_{\lambda}^{\theta_{\ell+1}b}       & \cdots & \theta_{2N}\,e_{\lambda}^{\theta_{2N}b}
\\[.5ex]
\,\vdots                                  &        & \vdots                                                & \vdots                                      & \vdots                                                &        & \vdots
\\[.5ex]
\,\theta_1^{N-1}\,e_{\lambda}^{\theta_1b} & \cdots & \theta_{\ell-1}^{N-1}\,e_{\lambda}^{\theta_{\ell-1}b} & \lambda^{-\frac{N-1}{2N}}\varphi^{(N-1)}(b) & \theta_{\ell+1}^{N-1}\,e_{\lambda}^{\theta_{\ell+1}b} & \cdots & \theta_{2N}^{N-1}\,e_{\lambda}^{\theta_{2N}b}
\end{vmatrix}.
\]
The determinant $\Delta_{\ell}(\lambda,\varphi)$ can be expanded
with respect to its $\ell$th column:
\[
\Delta_{\ell}(\lambda,\varphi)=\sum_{k=0}^{N-1} \lambda^{-\frac{k}{2N}}
\Delta_{k\ell}^-(\lambda)\,\varphi^{(k)}(a)+\sum_{k=0}^{N-1} \lambda^{-\frac{k}{2N}}
\Delta_{k\ell}^+(\lambda)\,\varphi^{(k)}(b)
\]
with
\begin{align*}
\Delta_{k\ell}^-(\lambda)
&
=\begin{vmatrix}
\,e_{\lambda}^{\theta_1a}                 & \cdots & e_{\lambda}^{\theta_{\ell-1}a}                        & 0      & e_{\lambda}^{\theta_{\ell+1}a}                        & \cdots & e_{\lambda}^{\theta_{2N}a}
\\[.5ex]
\,\vdots                                  &        & \vdots                                                & \vdots & \vdots                                                &        & \vdots
\\[.5ex]
\,\theta_1^{k-1}\,e_{\lambda}^{\theta_1a} & \cdots & \theta_{\ell-1}^{k-1}\,e_{\lambda}^{\theta_{\ell-1}a} & 0      & \theta_{\ell+1}^{k-1}\,e_{\lambda}^{\theta_{\ell+1}a} & \cdots & \theta_{2N}^{k-1}\,e_{\lambda}^{\theta_{2N}a}
\\[.5ex]
\,\theta_1^k\,e_{\lambda}^{\theta_1a}     & \cdots & \theta_{\ell-1}^k\,e_{\lambda}^{\theta_{\ell-1}a}     & 1      & \theta_{\ell+1}^k\,e_{\lambda}^{\theta_{\ell+1}a}     & \cdots & \theta_{2N}^k\,e_{\lambda}^{\theta_{2N}a}
\\[.5ex]
\,\theta_1^{k+1}\,e_{\lambda}^{\theta_1a} & \cdots & \theta_{\ell-1}^{k+1}\,e_{\lambda}^{\theta_{\ell-1}a} & 0      & \theta_{\ell+1}^{k+1}\,e_{\lambda}^{\theta_{\ell+1}a} & \cdots & \theta_{2N}^{k+1}\,e_{\lambda}^{\theta_{2N}a}
\\[.5ex]
\,\vdots                                  &        & \vdots                                                & \vdots & \vdots                                                &        & \vdots
\\[.5ex]
\,\theta_1^{N-1}\,e_{\lambda}^{\theta_1a} & \cdots & \theta_{\ell-1}^{N-1}\,e_{\lambda}^{\theta_{\ell-1}a} & 0      & \theta_{\ell+1}^{N-1}\,e_{\lambda}^{\theta_{\ell+1}a} & \cdots & \theta_{2N}^{N-1}\,e_{\lambda}^{\theta_{2N}a}
\\[-0.5ex]
\hdotsfor{7}
\\
\,e_{\lambda}^{\theta_1b}                 & \cdots & e_{\lambda}^{\theta_{\ell-1}b}                        & 0      & e_{\lambda}^{\theta_{\ell+1}b}                        & \cdots & e_{\lambda}^{\theta_{2N}b}
\\[.5ex]
\,\vdots                                  &        & \vdots                                                & \vdots & \vdots                                                &        & \vdots
\\[.5ex]
\,\theta_1^{N-1}\,e_{\lambda}^{\theta_1b} & \cdots & \theta_{\ell-1}^{N-1}\,e_{\lambda}^{\theta_{\ell-1}b} & 0      & \theta_{\ell+1}^{N-1}\,e_{\lambda}^{\theta_{\ell+1}b} & \cdots & \theta_{2N}^{N-1}\,e_{\lambda}^{\theta_{2N}b}
\end{vmatrix}
\\[3ex]
\lqn{}
&
=\begin{vmatrix}
\,e_{\lambda}^{\theta_1a}                 & \cdots & e_{\lambda}^{\theta_{\ell-1}a}                        & 0      & e_{\lambda}^{\theta_{\ell+1}a}                        & \cdots & e_{\lambda}^{\theta_{2N}a}
\\[.5ex]
\,\vdots                                  &        & \vdots                                                & \vdots & \vdots                                                &        & \vdots
\\[.5ex]
\,\theta_1^{k-1}\,e_{\lambda}^{\theta_1a} & \cdots & \theta_{\ell-1}^{k-1}\,e_{\lambda}^{\theta_{\ell-1}a} & 0      & \theta_{\ell+1}^{k-1}\,e_{\lambda}^{\theta_{\ell+1}a} & \cdots & \theta_{2N}^{k-1}\,e_{\lambda}^{\theta_{2N}a}
\\[.5ex]
\,0                                       & \cdots & 0                                                     & 1      & 0                                                     & \cdots &
\\[.5ex]
\,\theta_1^{k+1}\,e_{\lambda}^{\theta_1a} & \cdots & \theta_{\ell-1}^{k+1}\,e_{\lambda}^{\theta_{\ell-1}a} & 0      & \theta_{\ell+1}^{k+1}\,e_{\lambda}^{\theta_{\ell+1}a} & \cdots & \theta_{2N}^{k+1}\,e_{\lambda}^{\theta_{2N}a}
\\[.5ex]
\,\vdots                                  &        & \vdots                                                & \vdots & \vdots                                                &        & \vdots
\\[.5ex]
\,\theta_1^{N-1}\,e_{\lambda}^{\theta_1a} & \cdots & \theta_{\ell-1}^{N-1}\,e_{\lambda}^{\theta_{\ell-1}a} & 0      & \theta_{\ell+1}^{N-1}\,e_{\lambda}^{\theta_{\ell+1}a} & \cdots & \theta_{2N}^{N-1}\,e_{\lambda}^{\theta_{2N}a}
\\[-0.5ex]
\hdotsfor{7}
\\
\,e_{\lambda}^{\theta_1b}                 & \cdots & e_{\lambda}^{\theta_{\ell-1}b}                        & 0      & e_{\lambda}^{\theta_{\ell+1}b}                        & \cdots & e_{\lambda}^{\theta_{2N}b}
\\[.5ex]
\,\vdots                                  &        & \vdots                                                & \vdots & \vdots                                                &        & \vdots
\\[.5ex]
\,\theta_1^{N-1}\,e_{\lambda}^{\theta_1b} & \cdots & \theta_{\ell-1}^{N-1}\,e_{\lambda}^{\theta_{\ell-1}b} & 0      & \theta_{\ell+1}^{N-1}\,e_{\lambda}^{\theta_{\ell+1}b} & \cdots & \theta_{2N}^{N-1}\,e_{\lambda}^{\theta_{2N}b}
\end{vmatrix}
\\[3ex]
\lqn{}
&
=\begin{vmatrix}
\,e_{\lambda}^{\theta_1a}                 & \cdots & e_{\lambda}^{\theta_{\ell-1}a}                        & e_{\lambda}^{\theta_{\ell}a}                       & e_{\lambda}^{\theta_{\ell+1}a}                        & \cdots & e_{\lambda}^{\theta_{2N}a}
\\[.5ex]
\,\vdots                                  &        & \vdots                                                & \vdots                                             & \vdots                                                &        & \vdots
\\[.5ex]
\,\theta_1^{k-1}\,e_{\lambda}^{\theta_1a} & \cdots & \theta_{\ell-1}^{k-1}\,e_{\lambda}^{\theta_{\ell-1}a} & \theta_{\ell}^{k-1}\,e_{\lambda}^{\theta_{\ell}a}  & \theta_{\ell+1}^{k-1}\,e_{\lambda}^{\theta_{\ell+1}a} & \cdots & \theta_{2N}^{k-1}\,e_{\lambda}^{\theta_{2N}a}
\\[.5ex]
\,0                                       & \cdots & 0                                                     & 1                                                  & 0                                                     & \cdots &
\\[.5ex]
\,\theta_1^{k+1}\,e_{\lambda}^{\theta_1a} & \cdots & \theta_{\ell-1}^{k+1}\,e_{\lambda}^{\theta_{\ell-1}a} & \theta_{\ell}^{k+1}\,e_{\lambda}^{\theta_{\ell}a}  & \theta_{\ell+1}^{k+1}\,e_{\lambda}^{\theta_{\ell+1}a} & \cdots & \theta_{2N}^{k+1}\,e_{\lambda}^{\theta_{2N}a}
\\[.5ex]
\,\vdots                                  &        & \vdots                                                & \vdots                                             & \vdots                                                &        & \vdots
\\[.5ex]
\,\theta_1^{N-1}\,e_{\lambda}^{\theta_1a} & \cdots & \theta_{\ell-1}^{N-1}\,e_{\lambda}^{\theta_{\ell-1}a} & \theta_{\ell}^{N-1}\,e_{\lambda}^{\theta_{\ell}a}  & \theta_{\ell+1}^{N-1}\,e_{\lambda}^{\theta_{\ell+1}a} & \cdots & \theta_{2N}^{N-1}\,e_{\lambda}^{\theta_{2N}a}
\\[-0.5ex]
\hdotsfor{7}
\\
\,e_{\lambda}^{\theta_1b}                 & \cdots & e_{\lambda}^{\theta_{\ell-1}b}                        & e_{\lambda}^{\theta_{\ell}b}                       & e_{\lambda}^{\theta_{\ell+1}b}                        & \cdots & e_{\lambda}^{\theta_{2N}b}
\\[.5ex]
\,\vdots                                  &        & \vdots                                                & \vdots                                             & \vdots                                                &        & \vdots
\\[.5ex]
\,\theta_1^{N-1}\,e_{\lambda}^{\theta_1b} & \cdots & \theta_{\ell-1}^{N-1}\,e_{\lambda}^{\theta_{\ell-1}b} & \theta_{\ell}^{N-1}\,e_{\lambda}^{\theta_{\ell}b}  & \theta_{\ell+1}^{N-1}\,e_{\lambda}^{\theta_{\ell+1}b} & \cdots & \theta_{2N}^{N-1}\,e_{\lambda}^{\theta_{2N}b}
\end{vmatrix}
\end{align*}
and, in the same manner,
\begin{align*}
\Delta_{k\ell}^+(\lambda)
&
=\begin{vmatrix}
\,e_{\lambda}^{\theta_1a}                 & \cdots & e_{\lambda}^{\theta_{\ell-1}a}                        & e_{\lambda}^{\theta_{\ell}a}                       & e_{\lambda}^{\theta_{\ell+1}a}                        & \cdots & e_{\lambda}^{\theta_{2N}a}
\\[.5ex]
\,\vdots                                  &        & \vdots                                                & \vdots                                             & \vdots                                                &        & \vdots
\\[.5ex]
\,\theta_1^{N-1}\,e_{\lambda}^{\theta_1a} & \cdots & \theta_{\ell-1}^{N-1}\,e_{\lambda}^{\theta_{\ell-1}a} & \theta_{\ell}^{N-1}\,e_{\lambda}^{\theta_{\ell}a}  & \theta_{\ell+1}^{N-1}\,e_{\lambda}^{\theta_{\ell+1}a} & \cdots & \theta_{2N}^{N-1}\,e_{\lambda}^{\theta_{2N}a}
\\[-0.5ex]
\hdotsfor{7}
\\
\,e_{\lambda}^{\theta_1b}                 & \cdots & e_{\lambda}^{\theta_{\ell-1}b}                        & e_{\lambda}^{\theta_{\ell}b}                       & e_{\lambda}^{\theta_{\ell+1}b}                        & \cdots & e_{\lambda}^{\theta_{2N}b}
\\[.5ex]
\,\vdots                                  &        & \vdots                                                & \vdots                                             & \vdots                                                &        & \vdots
\\[.5ex]
\,\theta_1^{k-1}\,e_{\lambda}^{\theta_1b} & \cdots & \theta_{\ell-1}^{k-1}\,e_{\lambda}^{\theta_{\ell-1}b} & \theta_{\ell}^{k-1}\,e_{\lambda}^{\theta_{\ell}b}  & \theta_{\ell+1}^{k-1}\,e_{\lambda}^{\theta_{\ell+1}b} & \cdots & \theta_{2N}^{k-1}\,e_{\lambda}^{\theta_{2N}a}
\\[.5ex]
\,0                                       & \cdots & 0                                                     & 1                                                  & 0                                                     & \cdots &
\\[.5ex]
\,\theta_1^{k+1}\,e_{\lambda}^{\theta_1b} & \cdots & \theta_{\ell-1}^{k+1}\,e_{\lambda}^{\theta_{\ell-1}b} & \theta_{\ell}^{k+1}\,e_{\lambda}^{\theta_{\ell}b}  & \theta_{\ell+1}^{k+1}\,e_{\lambda}^{\theta_{\ell+1}b} & \cdots & \theta_{2N}^{k+1}\,e_{\lambda}^{\theta_{2N}a}
\\[.5ex]
\,\vdots                                  &        & \vdots                                                & \vdots                                             & \vdots                                                &        & \vdots
\\[.5ex]
\,\theta_1^{N-1}\,e_{\lambda}^{\theta_1b} & \cdots & \theta_{\ell-1}^{N-1}\,e_{\lambda}^{\theta_{\ell-1}b} & \theta_{\ell}^{N-1}\,e_{\lambda}^{\theta_{\ell}b}  & \theta_{\ell+1}^{N-1}\,e_{\lambda}^{\theta_{\ell+1}b} & \cdots & \theta_{2N}^{N-1}\,e_{\lambda}^{\theta_{2N}b}
\end{vmatrix}.
\end{align*}
With these settings at hand, we can write the solution of~(\ref{BVP1}):
\begin{align*}
\Phi(x)
&
=\sum_{\ell=1}^{2N}\alpha_{\ell} \,e_{\lambda}^{\theta_{\ell}x}
=\sum_{\ell=1}^{2N}\frac{\Delta_{\ell}(\lambda,\varphi)}{\Delta(\lambda)} \,e_{\lambda}^{\theta_{\ell}x}
\\
&
=\frac{1}{\Delta(\lambda)}\left[\sum_{k=0}^{N-1} \lambda^{-\frac{k}{2N}}
\left(\sum_{\ell=1}^{2N}\Delta_{k\ell}^-(\lambda)\,e_{\lambda}^{\theta_{\ell}x}\right)\varphi^{(k)}(a)
+\sum_{k=0}^{N-1} \lambda^{-\frac{k}{2N}}
\left(\sum_{\ell=1}^{2N}\Delta_{k\ell}^+(\lambda)\,e_{\lambda}^{\theta_{\ell}x}\right)\varphi^{(k)}(b)
\right]
\\
&
=\sum_{k=0}^{N-1} \lambda^{-\frac{k}{2N}}\frac{\Delta_k^-(\lambda;x)}{\Delta(\lambda)}\,\varphi^{(k)}(a)
+\sum_{k=0}^{N-1} \lambda^{-\frac{k}{2N}}\frac{\Delta_k^+(\lambda;x)}{\Delta(\lambda)}\,\varphi^{(k)}(b)
\end{align*}
with
\begin{equation}\label{delta}
\Delta_k^-(\lambda;x)
=\sum_{\ell=1}^{2N}\Delta_{k\ell}^-(\lambda)\,e_{\lambda}^{\theta_{\ell}x},\quad
\Delta_k^+(\lambda;x)
=\sum_{\ell=1}^{2N}\Delta_{k\ell}^+(\lambda)\,e_{\lambda}^{\theta_{\ell}x}.
\end{equation}
We immediately see that equalities~(\ref{delta}) are the expansions of
the determinants displayed in Theorem~\ref{theo-Phi} with respect to their
$(k-1)$th raw and $(k+N-1)$th raw respectively. Formula~(\ref{exp-Phi}) is proved.

Finally, it is easy to check the boundary value problems satisfied
by the functions $x\mapsto \Delta_k^{\pm}(\lambda;x)$ by using
elementary rules on differentiating a determinant.
In particular for, e.g., $\Delta_k^-$, the determinants defining
$(\Delta_k^-)^{(\ell)}(\lambda;a)$, $\ell\in\{0,\dots,N-1\}\backslash\{k\}$,
and $(\Delta_k^-)^{(\ell)}(\lambda;b)$, $\ell\in\{0,\dots,N-1\}$,
have two identical rows, thus they vanish.
The determinant $(\Delta_k^-)^{(k)}(\lambda;a)$ has the same rows as
$\Delta(\lambda)$ up to the multiplicative factor $\lambda^{k/(2N)}$
for its $k$th row, then it coincides with $\lambda^{k/(2N)}\Delta(\lambda)$.
The proof of Theorem~\ref{theo-Phi} is finished.
\end{proof}
%
Now, by eliminating the function $\varphi$ in (\ref{exp-Phi}),
we get the following result which should be understood in the sense of
Schwartz distributions:
\begin{align}
\lqn{\mathbb{E}_x\!\left(\mathrm{e}^{-\lambda\tau_{ab}}\,\ind_{\{\tau_{ab}<+\infty\}},
X_{\tau_{ab}}\in\mathrm{d}z\right)\!/\mathrm{d}z
}
=\sum_{k=0}^{N-1} (-1)^k\lambda^{-\frac{k}{2N}}
\frac{\Delta_k^-(\lambda;x)}{\Delta(\lambda)}\,\delta_a^{(k)}(z)
+\sum_{k=0}^{N-1} (-1)^k\lambda^{-\frac{k}{2N}}
\frac{\Delta_k^+(\lambda;x)}{\Delta(\lambda)}\,\delta_b^{(k)}(z)
\label{exp-joint-dist-inter}
\end{align}
from which we derive the following representation for
the pseudo-distribution of $\left(\tau_{ab},X_{\tau_{ab}}\right)$.
%
\begin{theorem}\label{theo-joint-dist-Xtau-ab}
The joint pseudo-distribution of $\left(\tau_{ab},X_{\tau_{ab}}\right)$
admits the following representation:
\begin{align}
\mathbb{P}_{\!x}\{\tau_{ab}\in\mathrm{d}t,X_{\tau_{ab}}\in\mathrm{d}z\}
/\mathrm{d}t\,\mathrm{d}z
&
=\sum_{k=0}^{N-1} (-1)^k I_k^-(t;x)\,\delta_a^{(k)}(z)
+\sum_{k=0}^{N-1} (-1)^k I_k^+(t;x)\,\delta_b^{(k)}(z)
\label{exp-joint-dist}
\end{align}
where the functions $I_k^{\pm}(t;x)$, $0\le k\le N-1$, are characterized by
their Laplace transforms:
\[
\int_0^{\infty} I_k^{\pm}(t;x) \,\mathrm{e}^{-\lambda t}\,\mathrm{d}t =
\lambda^{-\frac{k}{2N}}\frac{\Delta_k^{\pm}(\lambda;x)}{\Delta(\lambda)}.
\]
They are also characterized by the boundary value problems
\[
\left\{\!\!\begin{array}{l}
\displaystyle\frac{\partial I_k^-}{\partial t}(t;x)=\kappa_{_{\!N}}
\frac{\partial^{2N}\! I_k^-}{\partial^{2N}x}(t;x)
\\[2ex]
\displaystyle\frac{\partial^{\ell} I_k^-}{\partial^{\ell}x}(t;a)=\delta_{k\ell},
\;\frac{\partial^{\ell} I_k^-}{\partial^{\ell}x}(t;b)=0
\quad\text{for }\ell\in\{0,1,\dots,N-1\},
\end{array}\right.
\]
\[
\left\{\!\!\begin{array}{l}
\displaystyle\frac{\partial I_k^+}{\partial t}(t;x)=\kappa_{_{\!N}}
\frac{\partial^{2N}\! I_k^+}{\partial^{2N}x}(t;x)
\\[2ex]
\displaystyle\frac{\partial^{\ell} I_k^+}{\partial^{\ell}x}(t;a)=0,
\;\frac{\partial^{\ell} I_k^+}{\partial^{\ell}x}(t;b)=\delta_{k\ell}
\quad\text{for }\ell\in\{0,1,\dots,N-1\}.
\end{array}\right.
\]
\end{theorem}
%
The boundary value problems satisfied by the functions $I_k^{\pm}$,
$0\le k\le N-1$, come from those satisfied by the functions $\Delta_k^{\pm}$
displayed in Theorem~\ref{theo-Phi}. The only details we have to check are that
$I_k^{\pm}(t;x)$ goes to 0 as $t$ tends to $0^+$ and that
$I_k^{\pm}(t;x)$ is bounded as $t$ tends to $+\infty$
(in order to have
$
\int_0^{\infty} (\partial/\partial t) I_k^{\pm}(t;x)
\,\mathrm{e}^{-\lambda t}\,\mathrm{d}t
=\lambda\int_0^{\infty} I_k^{\pm}(t;x)\,\mathrm{e}^{-\lambda t}\,\mathrm{d}t
$)
which can be deduced from the fact that their Laplace transforms go to 0
exponentially quickly as $\lambda$ goes to $+\infty$
and are bounded as $\lambda$ goes to $0^+$. These facts are proved
in Appendix~\ref{appendix1}; see~(\ref{limits2}) and (\ref{asymptotics}).
%
\begin{remark}\label{remark-symmetry}\rm
The functions $I_k^{\pm}$, $0\le k\le N-1$, are real-valued.
Indeed, observing that the complex numbers $\theta_{\ell}$, $1\le\ell\le 2N$, are
conjugate two by two, it is easily seen that the determinants contain
conjugate columns two by two, so they are real numbers. More precisely,
conjugating $\theta_1,\dots,\theta_N,\theta_{N+1},
\dots,\theta_{2N}$ respectively yields $\theta_N,\dots,\theta_1,
\theta_{2N},\dots,\theta_{N+1}$. Therefore, conjugating the determinants
$\Delta$ and $\Delta_{k\ell}^{\pm}$ boils down to interchanging their $1$st and
$N$th columns, their $2$nd and $(N-1)$th columns, $\dots$, their $(N+1)$th
and $(2N)$th columns, their $(N+2)$th and $(2N-1)$th columns, and so on.
In this way, we perform an even number of transpositions and we retrieve
the original determinants: $\overline{\Delta}=\Delta$ and
$\overline{\Delta_{k\ell}^{\pm}}=\Delta_{k\ell}^{\pm}$, proving that they
are real numbers.

Moreover, the functions $I_k^+$ and $I_k^-$ are related according to
the identity $I_k^+(t;x)=(-1)^k$ $I_k^-(t;a+b-x)$ as it can be seen by proving
the same identity concerning their Laplace transforms; see~(\ref{symmetry})
in Appendix~\ref{appendix1}.

\end{remark}
%
\begin{remark}\label{remark-limit}\rm
Let us compute the limit of (\ref{exp-joint-dist-inter}) as $b$ tends towards
$+\infty$. To this aim, we find that
\begin{equation}\label{limits}
\frac{\Delta_k^-(\lambda;x)}{\Delta(\lambda)}\underset{b\to+\infty}{\longrightarrow}
\sum_{\ell=1}^{N} \alpha_{k\ell}\,\mathrm{e}_{\lambda}^{\theta_{\ell}\,(x-a)}
\quad\text{and}\quad
\frac{\Delta_k^+(\lambda;x)}{\Delta(\lambda)}\underset{b\to+\infty}{\longrightarrow}0
\end{equation}
with
\[
\alpha_{k\ell}=\frac{1}{\det(V)} \begin{vmatrix}
\,1              & \cdots & 1                     & 0      & 1                     & \cdots & 1
\\[.5ex]
\,\theta_1       & \cdots & \theta_{\ell-1}       & 0      & \theta_{\ell+1}       & \cdots & \theta_N
\\[.5ex]
\,\vdots         &        &\vdots                 & \vdots & \vdots                &        & \vdots
\\[.5ex]
\,\theta_1^{k-1} & \cdots & \theta_{\ell-1}^{k-1} & 0      & \theta_{\ell+1}^{k-1} & \cdots & \theta_N^{k-1}
\\[.5ex]
\,0              & \cdots & 0                     & 1      & 0                     & \cdots & 0
\\[.5ex]
\,\theta_1^{k+1} & \cdots & \theta_{\ell-1}^{k+1} & 0      & \theta_{\ell+1}^{k+1} & \cdots & \theta_N^{k+1}
\\[.5ex]
\,\vdots         &        & \vdots                & \vdots & \vdots                &        & \vdots
\\[.5ex]
\,\theta_1^{N-1} & \cdots & \theta_{\ell-1}^{N-1} & 0      & \theta_{\ell+1}^{N-1} & \cdots & \theta_N^{N-1}
\end{vmatrix}\!.
\]
The coefficients $\alpha_{k\ell}$ are characterized by the identity
\begin{equation}\label{pol-alpha}
\sum_{k=0}^{N-1} \alpha_{k\ell}\,x^k=
\frac{1}{\det(V)} \begin{vmatrix}
\,1              & \cdots & 1                     & 1       & 1                     & \cdots & 1
\\[.5ex]
\,\theta_1       & \cdots & \theta_{\ell-1}       & x       & \theta_{\ell+1}       & \cdots & \theta_N
\\[.5ex]
\,\vdots         &        &\vdots                 & \vdots  & \vdots                &        & \vdots
\\[.5ex]
\,\theta_1^{N-1} & \cdots & \theta_{\ell-1}^{N-1} & x^{N-1} & \theta_{\ell+1}^{N-1} & \cdots & \theta_N^{N-1}
\end{vmatrix}
=\prod_{1\le k\le N \atop k\neq \ell} \left(\frac{x-\theta_k}
{\theta_{\ell}-\theta_k}\right)
\end{equation}
as it is easily seen by appealing to the well-known Vandermonde
determinant $\det(V)=\prod_{1\le i<j\le N}$ $(\theta_j-\theta_i)$.
Notice that polynomial~(\ref{pol-alpha}) is nothing but an elementary
Lagrange interpolating polynomial related to the numbers $\theta_i$, $1\le i\le N$.
The details of these limiting results being cumbersome, we postpone them
to Appendix~\ref{appendix1}.

In regards to (\ref{exp-joint-dist-inter}), (\ref{exp-joint-dist}) and
(\ref{limits}), we conclude that, for $x>a$,
\[
\lim_{b\to +\infty}\mathbb{P}_{\!x}\{\tau_{ab}\in\mathrm{d}t,
X_{\tau_{ab}}\in\mathrm{d}z\}/\mathrm{d}t\,\mathrm{d}z
=\sum_{k=0}^{N-1} (-1)^k K_k(t;x)\,\delta_a^{(k)}(z)
\]
where $K_k$ is the function whose Laplace transform is given by
\[
\int_0^{\infty} K_k(t;x) \,\mathrm{e}^{-\lambda t}\,\mathrm{d}t
=\lambda^{-\frac{k}{2N}}\sum_{\ell=1}^{N} \alpha_{k\ell}
\,\mathrm{e}^{\theta_{\ell}\sqrt[2N]{\lambda}\,(x-a)}.
\]
We retrieve at the limit the pseudo-distribution of $(\tau_a,X_{\tau_a})$
related to the first overshooting time of level $a$ displayed in~\cite{la3},
formula~(5.15).
\end{remark}
%

\section{Pseudo-distribution of $\tau_{ab}$}

By applying the Schwartz distribution~(\ref{exp-joint-dist}) to the test
function~$1$, we immediately extract the pseudo-distribution of $\tau_{ab}$:
$\mathbb{P}_{\!x}\{\tau_{ab}\in\mathrm{d}t\}/\mathrm{d}t=I_0^-(t;x)+I_0^+(t;x)$
that we state as follows.
%
\begin{theorem}\label{theo-dist-tau-ab}
The pseudo-distribution of $\tau_{ab}$ is given either by one of
both formulae below:
\[
\mathbb{P}_{\!x}\{\tau_{ab}\in\mathrm{d}t\}/\mathrm{d}t=I(t;x),
\quad
\mathbb{P}_{\!x}\{\tau_{ab}\le t\}=J(t;x) 
\]
with
\[
\int_0^{\infty} I(t;x) \,\mathrm{e}^{-\lambda t}\,\mathrm{d}t =
\frac{\Delta_0^+(\lambda;x)+\Delta_0^-(\lambda;x)}{\Delta(\lambda)},
\quad
\int_0^{\infty} J(t;x) \,\mathrm{e}^{-\lambda t}\,\mathrm{d}t =
\frac{1}{\lambda} \,\frac{\Delta_0^+(\lambda;x)+\Delta_0^-(\lambda;x)}{\Delta(\lambda)}.
\]
\end{theorem}
%
Let us introduce the up-to-date minimum and maximum functionals of $X$:
\[
m_t=\min_{s\in[0,t]}X_s,\quad M_t=\max_{s\in[0,t]}X_s.
\]
It is plain that the functionals $m_t,M_t$ and time $\tau_{ab}$ are related
according as $a<m_t\le M_t<b\Longleftrightarrow \tau_{ab}>t$. Then
$\mathbb{P}_{\!x}\{a<m_t\le M_t<b\}=1-\mathbb{P}_{\!x}\{\tau_{ab}\le t\}.$
%
\begin{corollary}
The joint pseudo-distribution of $(m_t,M_t)$ is given by
\[
\mathbb{P}_{\!x}\{a<m_t\le M_t<b\}=1-J(t;x)
\]
and its Laplace transform with respect to $t$ writes
\[
\int_0^{\infty} \mathbb{P}_{\!x}\{a<m_t\le M_t<b\}\,\mathrm{e}^{-\lambda t}
\,\mathrm{d}t =\frac{1}{\lambda}\,
\frac{\Delta(\lambda)-\Delta_0^+(\lambda;x)-\Delta_0^-(\lambda;x)}{\Delta(\lambda)}.
\]
\end{corollary}
%

\section{Pseudo-distribution of $X_{\tau_{ab}}$}

In this part, we focus on the exit location of $X$ at time $\tau_{ab}$
whose pseudo-distribution admits a remarkable expression by means of Hermite
interpolating polynomials whose expressions are displayed in the introduction.
%
\begin{theorem}\label{theo-dist-Xab}
The pseudo-distribution of the exit location $X_{\tau_{ab}}
\ind_{\{\tau_{ab}<+\infty\}}$ is given, in the sense of Schwartz distributions, by
\begin{equation}\label{exp-location-dist}
\mathbb{P}_{\!x}\{X_{\tau_{ab}}\in \mathrm{d} z,\tau_{ab}<+\infty\}/\mathrm{d} z
=\sum_{k=0}^{N-1} (-1)^k H_k^-(x)\,\delta_a^{(k)}(z)
+\sum_{k=0}^{N-1} (-1)^k H_k^+(x)\,\delta_b^{(k)}(z).
\end{equation}
\end{theorem}
%
\begin{proof}
We directly solve boundary value problem~(\ref{BVP1}) in the case where
$\lambda=0$ therein. Namely, by setting
$\Psi(x)=\mathbb{E}_x\!\left(\varphi(X_{\tau_{ab}})\,\ind_{\{\tau_{ab}<+\infty\}}\right)$,
\[
\left\{\!\!\begin{array}{l}
\Psi^{(2N)}(x)=0,\quad x\in(a,b),
\\[1ex]
\Psi^{(k)}(a)=\varphi^{(k)}(a)\text{ and }
\Psi^{(k)}(b)=\varphi^{(k)}(b) \quad \text{for } k\in\{0,1,\dots,N-1\}.
\end{array}\right.
\]
It is clear that $\Psi$ is the polynomial of degree not greater than $(2N-1)$
whose derivatives at $a$ and $b$ up to order $(N-1)$ are
the given numbers $\varphi^{(k)}(a)$ and $\varphi^{(k)}(b)$, $0\le k\le N-1$.
It can be written as a linear combination of the Hermite interpolating
fundamental polynomials $H_k^{\pm}$, $0\le k\le N-1$, displayed in
Theorem~\ref{theo-dist-Xab} as follows: for any test functions $\varphi$,
\begin{equation}\label{BVP2}
\mathbb{E}_x\!\left(\varphi(X_{\tau_{ab}})\,\ind_{\{\tau_{ab}<+\infty\}}\right)
=\sum_{k=0}^{N-1} H_k^-(x) \,\varphi^{(k)}(a)
+\sum_{k=0}^{N-1} H_k^+(x) \,\varphi^{(k)}(b).
\end{equation}
Formula~(\ref{exp-location-dist}) is nothing but (\ref{BVP2}) rephrased by means of
Schwartz distributions.
\end{proof}
%
\begin{remark}\label{remark-expectation-hermite}\rm
Formula~(\ref{BVP2}) yields for $\varphi=H_k^{\pm}$, $0\le k\le N-1$, that
$\mathbb{E}_x\!\left(H_k^{\pm}(X_{\tau_{ab}})\,\ind_{\{\tau_{ab}<+\infty\}}\right)
=H_k^{\pm}(x)$.
\end{remark}
%
\begin{remark}
By letting $b$ tend to $+\infty$, we see that $H_k^+(x)$ tends to $0$
while $H_k^-(x)$ tends to $(x-a)^k/(k\,!)$. Hence, we find that
\[
\lim_{b\to +\infty}\mathbb{P}_{\!x}\{X_{\tau_{ab}}\in\mathrm{d}z,
\tau_{ab}<+\infty\}/\mathrm{d}z
=\sum_{k=0}^{N-1} \frac{(a-x)^k}{k!} \,\delta_a^{(k)}(z).
\]
We retrieve at the limit the pseudo-distribution~(\ref{dirac1})
of the location $X_{\tau_a}$ of $X$ at the first overshooting time of level $a$,
which is displayed in~\cite{la3}, formula~(5.18).
\end{remark}
%
\begin{corollary}
Time $\tau_{ab}$ is $\mathbb{P}_{\!x}$-almost surely finite in the sense that
\[
\mathbb{P}_{\!x}\{\tau_{ab}<+\infty\}=1.
\]
\end{corollary}
%
Because of this, in the sequel of the paper, we shall omit the condition
$\tau_{ab}<+\infty$ when considering the pseudo-random variable $X_{\tau_{ab}}$.
Actually, let us recall that, in the framework of signed measures,
if $A$ is a set of $\mathbb{P}_{\!x}$-measure~$1$,
it does \textit{not} entail that for any set $B$ that $\mathbb{P}_{\!x}(A\cap B)=
\mathbb{P}_{\!x}(B)$ contrarily to the case of ordinary probability.

\begin{proof}
The pseudo-probability $\mathbb{P}\{\tau_{ab}<+\infty\}$
can be deduced from~(\ref{exp-location-dist}) by choosing $\varphi=1$.
Indeed, we have that
\begin{align*}
H_0^-(x)
&=\left(\frac{b-x}{b-a}\right)^{\!\!N} \sum_{\ell=0}^{N-1}\binom{\ell+N-1}{\ell}
\!\left(\frac{x-a}{b-a}\right)^{\!\!\ell}
\\
&
=\frac{(b-x)^N}{(b-a)^{2N-1}} \sum_{\ell=0}^{N-1}\binom{\ell+N-1}{\ell} (x-a)^\ell(b-a)^{N-1-\ell}.
\end{align*}
By writing the term $(b-a)^{N-1-\ell}$ as
\[
(b-a)^{N-1-\ell}=[(x-a)+(b-x)]^{N-1-\ell}
=\sum_{k=0}^{N-1-\ell}\binom{N-1-\ell}{k}(x-a)^k(b-x)^{N-1-k-\ell},
\]
it follows that
\begin{align*}
H_0^-(x)
&=\frac{1}{(b-a)^{2N-1}} \sum_{0\le\ell\le N-1 \atop 0\le k\le N-1-\ell}
\binom{N-1-\ell}{k}\!\binom{\ell+N-1}{\ell}  (x-a)^{k+\ell}(b-x)^{2N-1-k-\ell}
\\
&
=\frac{1}{(b-a)^{2N-1}} \sum_{m=0}^{N-1} \left[\,\sum_{\ell=0}^m
\binom{\ell+N-1}{\ell} \!\binom{N-1-\ell}{m-\ell}\!\right] (x-a)^m (b-x)^{2N-1-m}.
\end{align*}
By using the elementary identity
$\sum_{\ell=0}^n \binom{\ell+p}{\ell} \!\binom{n+q-\ell}{n-\ell}=\binom{n+p+q+1}{n}$
which comes from the equality $(1+u)^{-p}(1+u)^{-q}=(1+u)^{-p-q}$
together with the expansion, e.g., for $p$, $(1+u)^{-p}=\sum_{\ell=0}^{\infty}
(-1)^\ell\binom{\ell+p-1}{\ell} u^\ell$, we get that
\[
\sum_{\ell=0}^m \binom{\ell+N-1}{\ell} \!\binom{N-1-\ell}{m-\ell}=\binom{2N-1}{m}.
\]
As a byproduct,
\[
H_0^-(x)=\frac{1}{(b-a)^{2N-1}} \sum_{m=0}^{N-1}\binom{2N-1}{m}
(x-a)^m (b-x)^{2N-1-m}.
\]
Similarly,
\[
H_0^+(x)=\frac{1}{(b-a)^{2N-1}} \sum_{m=N}^{2N-1}\binom{2N-1}{m}
(x-a)^m (b-x)^{2N-1-m}
\]
and we immediately deduce that
\[
\mathbb{P}\{\tau_{ab}<+\infty\}=H_0^-(x)+H_0^+(x)
= \frac{1}{(b-a)^{2N-1}} \sum_{m=0}^{2N-1}\binom{2N-1}{m} (x-a)^m (b-x)^{2N-1-m}=1.
\]
\end{proof}
%

Let us introduce the first down- and up-overshooting times of the
single thresholds $a$ and $b$ for $(X_t)_{t\ge 0}$:
\[
\tau_a^-=\inf\{t\ge 0: X_t<a\},\quad \tau_b^+=\inf\{t\ge 0: X_t>b\}.
\]
The famous problem of the ruin of the gambler in the context of
pseudo-Brownian motion consists in computing the pseudo-probability
of overshooting one level ($a$ or $b$) before the other one. For instance,
we have that
\[
\mathbb{P}_{\!x}\{\tau_a^-<\tau_b^+\}=\mathbb{P}_{\!x}\{X_{\tau_{ab}}\le a\}.
\]
Hence, in view of formula~(\ref{exp-location-dist}), we obtain the following
result.
%
\begin{corollary}\label{cor-ruin}
The ``ruin'' pseudo-probabilities related to pseudo-Brownian motion
are given by
\[
\mathbb{P}_{\!x}\{\tau_a^-<\tau_b^+\}=H_0^-(x),\quad
\mathbb{P}_{\!x}\{\tau_b^+<\tau_a^-\}=H_0^+(x).
\]
\end{corollary}
%

In the corollary below, we provide a way for computing the
pseudo-moments of $X_{\tau_{ab}}$.
%
\begin{corollary}
Let $P$ be a polynomial and $R$ the remainder of the
Euclidean division of $P(x)$ by $(x-a)^N(x-b)^N$.
We have that
\[
\mathbb{E}_x[P(X_{\tau_{ab}})]=R(x).
\]
In particular, the pseudo-moments of $X_{\tau_{ab}}$ are given,
for any $p\in\{0,1,\dots,2N-1\}$, by
\[
\mathbb{E}_x[(X_{\tau_{ab}})^p]=x^p
\]
and for any positive integer $p$, by setting
$c_n=\sum_{k=0}^n \binom{N+k-1}{k}\binom{N+n-1-k}{n-k}a^k\,b^{n-k}$, by
\begin{align*}
\mathbb{E}_x[(X_{\tau_{ab}})^{2N+p})]
&
=x^{2N+p}-\left(\sum_{n=0}^p c_{p-n} x^n\right)\!(x-a)^N(x-b)^N
\\
&
=x^{2N+p}-\left[x^p+N(a+b)\,x^{p-1}\right.
\\
&
\hphantom{=\;}\left.
+\left(\textstyle\frac12\,N(N+1)(a^2+b^2)+N^2ab\right)x^{p-2}+\cdots
\right]\!(x-a)^N(x-b)^N.
\end{align*}
For instance,
\begin{align*}
\mathbb{E}_x[(X_{\tau_{ab}})^{2N}] & =x^{2N}-(x-a)^N(x-b)^N,
\\
\mathbb{E}_x[(X_{\tau_{ab}})^{2N+1}] & =x^{2N+1}-[x+N(a+b)](x-a)^N(x-b)^N,
\\
\mathbb{E}_x[(X_{\tau_{ab}})^{2N+2}] &
=x^{2N+2}-\left[x^2+N(a+b)\,x+\left(\textstyle\frac12\,N(N+1)(a^2+b^2)+N^2ab
\right)\!\right]\!(x-a)^N(x-b)^N.
\end{align*}
\end{corollary}
%
\begin{proof}
Let us introduce the quotient $Q$ of the
Euclidean division of $P(x)$ by $(x-a)^N(x-b)^N$:
we have $P(x)=Q(x)(x-a)^N(x-b)^N+R(x)$. The polynomial $R$ is of
degree not greater than $(2N-1)$.
Since $a$ and $b$ are roots of the polynomial $P(x)-R(x)=Q(x)(x-a)^N(x-b)^N$
with a multiplicity not less than $N$, the successive derivatives
of $P-R$ up to order $(N-1)$ vanish at $a$ and $b$. Therefore,
by~(\ref{BVP2}), we deduce that
\[
\mathbb{E}_x\!\left[Q(X_{\tau_{ab}})(X_{\tau_{ab}}-a)^N(X_{\tau_{ab}}-b)^N\right]\!=0
\]
and then
\[
\mathbb{E}_x[P(X_{\tau_{ab}})]=\mathbb{E}_x[R(X_{\tau_{ab}})].
\]
Since the polynomial $R$ is of degree not greater than $(2N-1)$,
we can write the decomposition
\[
R(x)=\sum_{k=0}^{N-1} R^{(k)}(a)\,H_k^-(x)+\sum_{k=0}^{N-1} R^{(k)}(b)\,H_k^+(x).
\]
Therefore, appealing to Remark~\ref{remark-expectation-hermite}, we obtain
that
\[
\mathbb{E}_x[P(X_{\tau_{ab}})]=\mathbb{E}_x[R(X_{\tau_{ab}})]
=R(x)=P(x)-Q(x)(x-a)^N(x-b)^N.
\]
Next, we compute the quotient $Q$ when $P(x)=x^{2N+p}$:
\begin{align*}
\frac{x^{2N+p}}{(x-a)^N(x-b)^N}
&
=x^p\left(1-\frac{a\vphantom{b}}{x}\right)^{\!-N}\!\left(1-\frac bx\right)^{\!-N}
\\
&
=x^p \left(\sum_{k=0}^{\infty} \binom{N+k-1}{k} \frac{a^k}{x^k}\right)\!\!
\left(\sum_{\ell=0}^{\infty} \binom{N+\ell-1}{\ell} \frac{b^{\ell}}{x^{\ell}}\right)
\\
&
=\sum_{k=0}^{\infty} c_n x^{p-n}
=\sum_{k=0}^p c_{p-n} x^n+\sum_{n=1}^{\infty} \frac{c_{n+p}}{x^{n}}
\end{align*}
where
\[
c_n=\sum_{k,\ell\ge 0\atop k+\ell=n} \binom{N+k-1}{k}\!\binom{N+\ell-1}{\ell}a^k\,b^{\ell}
=\sum_{k=0}^n \binom{N+k-1}{k}\!\binom{N+n-1-k}{n-k}a^k\,b^{n-k}.
\]
Then, the quotient of $x^{2N+p}$ by $(x-a)^N(x-b)^N$ is equal to
$\sum_{k=0}^n c_n x^{p-n}$. In particular,
\[
c_0=1,\;c_1=N(a+b),\;c_2=\frac12\,N(N+1)(a^2+b^2)+N^2ab.
\]

\end{proof}
%
\begin{remark}\rm
By~(\ref{BVP2}), we easily get that
\[
\mathbb{E}_x\!\left[(X_{\tau_{ab}}-b)^p\ind_{\{\tau_b^+<\tau_a^-\}}\right]=
\left\{\begin{array}{ll}
p\,! \,H_p^+(x) & \text{if } p\le N,
\\
0 & \text{if } p\ge N+1.
\end{array}\right.
\]
This formula suggests the following interpretation of Hermite polynomials in terms
of pseudo-Brownian motion: for $p\in\{0,\dots,N-1\}$,
\[
H_p^+(x)=\frac{1}{p\,!}\,\mathbb{E}_x\!\left[(X_{\tau_{ab}}-b)^p\ind_{\{\tau_b^+<\tau_a^-\}}\right]\!.
\]
\end{remark}
%

\section{The case $N=2$}

For $N=2$, pseudo-Brownian motion is the so-called biharmonic-pseudo-process.
In this case, the settings write
$\theta_1=\mathrm{e}^{\mathrm{i}\,3\pi\!/4}$,
$\theta_2=\mathrm{e}^{\mathrm{i}\,5\pi\!/4}=\overline{\theta_1\!}\,$,
$\theta_3=\mathrm{e}^{\mathrm{i}\,7\pi\!/4}$,
$\theta_4=\mathrm{e}^{\mathrm{i}\,\pi\!/4}=\overline{\theta_3\!}\,$,
and, by setting $\nu=\lambda/4$,
\[
\Delta(\lambda)=\begin{vmatrix}
\,e_{\lambda}^{\theta_1a}           & e_{\lambda}^{\theta_2a}           & e_{\lambda}^{\theta_3a}           & e_{\lambda}^{\theta_4a}
\\[.5ex]
\,\theta_1\,e_{\lambda}^{\theta_1a} & \theta_2\,e_{\lambda}^{\theta_2a} & \theta_3\,e_{\lambda}^{\theta_3a} & \theta_4\,e_{\lambda}^{\theta_4a}
\\[.5ex]
\,e_{\lambda}^{\theta_1b}           & e_{\lambda}^{\theta_2b}           & e_{\lambda}^{\theta_3b}           & e_{\lambda}^{\theta_4b}
\\[.5ex]
\,\theta_1\,e_{\lambda}^{\theta_1b} & \theta_2\,e_{\lambda}^{\theta_2b} & \theta_3\,e_{\lambda}^{\theta_3b} & \theta_4\,e_{\lambda}^{\theta_4b}
\end{vmatrix},
\]
\[
\Delta_0^-(\lambda;x)
=\begin{vmatrix}
\,e_{\lambda}^{\theta_1x}           & e_{\lambda}^{\theta_2x}           & e_{\lambda}^{\theta_3x}           & e_{\lambda}^{\theta_4x}
\\[.5ex]
\,\theta_1\,e_{\lambda}^{\theta_1a} & \theta_2\,e_{\lambda}^{\theta_2a} & \theta_3\,e_{\lambda}^{\theta_3a} & \theta_4\,e_{\lambda}^{\theta_4a}
\\[.5ex]
\,e_{\lambda}^{\theta_1b}           & e_{\lambda}^{\theta_2b}           & e_{\lambda}^{\theta_3b}           & e_{\lambda}^{\theta_4b}
\\[.5ex]
\,\theta_1\,e_{\lambda}^{\theta_1b} & \theta_2\,e_{\lambda}^{\theta_2b} & \theta_3\,e_{\lambda}^{\theta_3b} & \theta_4\,e_{\lambda}^{\theta_4b}
\end{vmatrix},\quad
\Delta_1^-(\lambda;x)
=\begin{vmatrix}
\,e_{\lambda}^{\theta_1a}           & e_{\lambda}^{\theta_2a}           & e_{\lambda}^{\theta_3a}           & e_{\lambda}^{\theta_4a}
\\[.5ex]
\,e_{\lambda}^{\theta_1x}           & e_{\lambda}^{\theta_2x}           & e_{\lambda}^{\theta_3x}           & e_{\lambda}^{\theta_4x}
\\[.5ex]
\,e_{\lambda}^{\theta_1b}           & e_{\lambda}^{\theta_2b}           & e_{\lambda}^{\theta_3b}           & e_{\lambda}^{\theta_4b}
\\[.5ex]
\,\theta_1\,e_{\lambda}^{\theta_1b} & \theta_2\,e_{\lambda}^{\theta_2b} & \theta_3\,e_{\lambda}^{\theta_3b} & \theta_4\,e_{\lambda}^{\theta_4b}
\end{vmatrix},
\]
\[
\Delta_0^+(\lambda;x)
=\begin{vmatrix}
\,e_{\lambda}^{\theta_1a}           & e_{\lambda}^{\theta_2a}           & e_{\lambda}^{\theta_3a}           & e_{\lambda}^{\theta_4a}
\\[.5ex]
\,\theta_1\,e_{\lambda}^{\theta_1a} & \theta_2\,e_{\lambda}^{\theta_2a} & \theta_3\,e_{\lambda}^{\theta_3a} & \theta_4\,e_{\lambda}^{\theta_4a}
\\[.5ex]
\,e_{\lambda}^{\theta_1x}           & e_{\lambda}^{\theta_2x}           & e_{\lambda}^{\theta_3x}           & e_{\lambda}^{\theta_4x}
\\[.5ex]
\,\theta_1\,e_{\lambda}^{\theta_1b} & \theta_2\,e_{\lambda}^{\theta_2b} & \theta_3\,e_{\lambda}^{\theta_3b} & \theta_4\,e_{\lambda}^{\theta_4b}
\end{vmatrix},\quad
\Delta_1^+(\lambda;x)
=\begin{vmatrix}
\,e_{\lambda}^{\theta_1a}           & e_{\lambda}^{\theta_2a}           & e_{\lambda}^{\theta_3a}           & e_{\lambda}^{\theta_4a}
\\[.5ex]
\,\theta_1\,e_{\lambda}^{\theta_1a} & \theta_2\,e_{\lambda}^{\theta_2a} & \theta_3\,e_{\lambda}^{\theta_3a} & \theta_4\,e_{\lambda}^{\theta_4a}
\\[.5ex]
\,e_{\lambda}^{\theta_1b}           & e_{\lambda}^{\theta_2b}           & e_{\lambda}^{\theta_3b}           & e_{\lambda}^{\theta_4b}
\\[.5ex]
\,e_{\lambda}^{\theta_1x}           & e_{\lambda}^{\theta_2x}           & e_{\lambda}^{\theta_3x}           & e_{\lambda}^{\theta_4x}
\end{vmatrix}.
\]
Elementary computations yield that
\[
\Delta(\lambda)=4\left[\cosh\!\big(2\sqrt[4]{\nu}\,(b-a)\big)
+\cos\!\big(2\sqrt[4]{\nu}\,(b-a)\big)-2\right]\!.
\]
Let us expand, e.g., $\Delta_0^-(\lambda;x)$ with respect to its first row:
\[
\Delta_0^-(\lambda;x)=c_1\,e_{\lambda}^{\theta_1x}+ c_2\,e_{\lambda}^{\theta_2x}
+c_3\,e_{\lambda}^{\theta_3x}+c_4\,e_{\lambda}^{\theta_4x}
\]
where $c_1,c_2,c_3,c_4$ are the cofactors of $\Delta_0^-(\lambda;x)$ related
to the first row. Straightforward (but cumbersome) computations yield
that $c_2=\overline{c_1\!}\,$ and $c_4=\overline{c_3\!}\,$ and
\begin{align*}
c_1
&
=(1-\mathrm{i})\,\mathrm{e}^{\sqrt[4]{\nu}\,((2b-a)-\mathrm{i}a)}
+(1+\mathrm{i})\,\mathrm{e}^{\sqrt[4]{\nu}\,(a-\mathrm{i}(2b-a))}
-2\,\mathrm{e}^{(1-\mathrm{i})\sqrt[4]{\nu}\,a},
\\
c_3
&
=(1-\mathrm{i})\,\mathrm{e}^{-\sqrt[4]{\nu}\,((2b-a)-\mathrm{i}a)}
+(1+\mathrm{i})\,\mathrm{e}^{-\sqrt[4]{\nu}\,(a-\mathrm{i}(2b-a))}
-2\,\mathrm{e}^{-(1-\mathrm{i})\sqrt[4]{\nu}\,a}.
\end{align*}
Therefore, we have that
\begin{align*}
\Delta_0^-(\lambda;x)
&
=2\,\Re\big(c_1\,e_{\lambda}^{\theta_1x}+c_3\,e_{\lambda}^{\theta_3x}\big)
\\
&
=2\left[\mathrm{e}^{\sqrt[4]{\nu}\,(x-a)}\cos\!\big(\sqrt[4]{\nu}\,(x+a-2b)\big)
+\mathrm{e}^{-\sqrt[4]{\nu}\,(x-a)}\cos\!\big(\sqrt[4]{\nu}\,(x+a-2b)\big)\right.
\\
&
\hphantom{=\;}
+\mathrm{e}^{\sqrt[4]{\nu}\,(x-a)}\sin\!\big(\sqrt[4]{\nu}\,(x+a-2b)\big)
-\mathrm{e}^{-\sqrt[4]{\nu}\,(x-a)}\sin\!\big(\sqrt[4]{\nu}\,(x+a-2b)\big)
\\
&
\hphantom{=\;}
-2\,\mathrm{e}^{\sqrt[4]{\nu}\,(x-a)}\cos\!\big(\sqrt[4]{\nu}\,(x-a)\big)
-2\,\mathrm{e}^{-\sqrt[4]{\nu}\,(x-a)}\cos\!\big(\sqrt[4]{\nu}\,(x-a)\big)
\\
&
\hphantom{=\;}
+\mathrm{e}^{\sqrt[4]{\nu}\,(x+a-2b)}\cos\!\big(\sqrt[4]{\nu}\,(x-a)\big)
+\mathrm{e}^{-\sqrt[4]{\nu}\,(x+a-2b)}\cos\!\big(\sqrt[4]{\nu}\,(x-a)\big)
\\
&
\left.\hphantom{=\;}
-\mathrm{e}^{\sqrt[4]{\nu}\,(x+a-2b)}\sin\!\big(\sqrt[4]{\nu}\,(x-a)\big)
+\mathrm{e}^{-\sqrt[4]{\nu}\,(x+a-2b)}\sin\!\big(\sqrt[4]{\nu}\,(x-a)\big)\right]
\end{align*}
which simplifies by means of hyperbolic functions into
\begin{align*}
\Delta_0^-(\lambda;x)
&
=4\left[\cosh\!\big(\sqrt[4]{\nu}\,(x-a)\big)\cos\!\big(\sqrt[4]{\nu}\,(x+a-2b)\big)
+\sinh\!\big(\sqrt[4]{\nu}\,(x-a)\big)\sin\!\big(\sqrt[4]{\nu}\,(x+a-2b)\big)\right.
\\
&
\hphantom{=\;}
+\cosh\!\big(\sqrt[4]{\nu}\,(x+a-2b)\big)\cos\!\big(\sqrt[4]{\nu}\,(x-a)\big)
-\sinh\!\big(\sqrt[4]{\nu}\,(x+a-2b)\big)\sin\!\big(\sqrt[4]{\nu}\,(x-a)\big)
\\
&
\left.\hphantom{=\;}
-2\cosh\!\big(\sqrt[4]{\nu}\,(x-a))\cos\!\big(\sqrt[4]{\nu}\,(x-a)\big)\right]\!.
\end{align*}
Quite similar computations yield that
\begin{align*}
\Delta_1^-(\lambda;x)
&
=4\left[\cosh\!\big(\sqrt[4]{\nu}\,(x+a-2b)\big)\sin\!\big(\sqrt[4]{\nu}\,(x-a)\big)
+\sinh\!\big(\sqrt[4]{\nu}\,(x-a)\big)\cos\!\big(\sqrt[4]{\nu}\,(x+a-2b)\big)\right.
\\
&
\left.\hphantom{=\;}
-\cosh\!\big(\sqrt[4]{\nu}\,(x-a)\big)\sin\!\big(\sqrt[4]{\nu}\,(x-a)\big)
-\sinh\!\big(\sqrt[4]{\nu}\,(x-a)\big)\cos\!\big(\sqrt[4]{\nu}\,(x-a)\big)\right]\!.
\end{align*}
The determinants $\Delta_0^+$ and $\Delta_1^+$ can be immediately deduced from
$\Delta_0^-$ and $\Delta_1^-$ by interchanging the roles of $a$ and $b$ as
it can be seen upon interchanging certain rows therein. We obtain that
\begin{align*}
\Delta_0^+(\lambda;x)
&
=4\left[\cosh\!\big(\sqrt[4]{\nu}\,(x-b)\big)\cos\!\big(\sqrt[4]{\nu}\,(x+b-2a)\big)
+\sinh\!\big(\sqrt[4]{\nu}\,(x-b)\big)\sin\!\big(\sqrt[4]{\nu}\,(x+b-2a)\big)\right.
\\
&
\hphantom{=\;}
+\cosh\!\big(\sqrt[4]{\nu}\,(x+b-2a)\big)\cos\!\big(\sqrt[4]{\nu}\,(x-b)\big)
-\sinh\!\big(\sqrt[4]{\nu}\,(x+b-2a)\big)\sin\!\big(\sqrt[4]{\nu}\,(x-b)\big)
\\
&
\left.\hphantom{=\;}
-2\cosh\!\big(\sqrt[4]{\nu}\,(x-b))\cos\!\big(\sqrt[4]{\nu}\,(x-b)\big)\right]\!,
\\
\Delta_1^+(\lambda;x)
&
=4\left[\cosh\!\big(\sqrt[4]{\nu}\,(x+b-2a)\big)\sin\!\big(\sqrt[4]{\nu}\,(x-b)\big)
+\sinh\!\big(\sqrt[4]{\nu}\,(x-b)\big)\cos\!\big(\sqrt[4]{\nu}\,(x+b-2a)\big)\right.
\\
&
\left.\hphantom{=\;}
-\cosh\!\big(\sqrt[4]{\nu}\,(x-b)\big)\sin\!\big(\sqrt[4]{\nu}\,(x-b)\big)
-\sinh\!\big(\sqrt[4]{\nu}\,(x-b)\big)\cos\!\big(\sqrt[4]{\nu}\,(x-b)\big)\right]\!.
\end{align*}
Now, formula~(\ref{exp-joint-dist}) reads
\[
\mathbb{P}_{\!x}\{\tau_{ab}\in\mathrm{d}t,X_{\tau_{ab}}\in\mathrm{d}z\}
/\mathrm{d}t\,\mathrm{d}z
=I_0^-(t;x)\,\delta_a(z) +I_1^-(t;x)\,\delta_a'(z)
+I_0^+(t;x)\,\delta_b(z) +I_1^+(t;x)\,\delta_b'(z)
\]
where the functions $I_0^{\pm}$ and $I_1^{\pm}$ are characterized by
\[
\int_0^{\infty} I_0^{\pm}(t;x) \,\mathrm{e}^{-\lambda t}\,\mathrm{d}t =
\frac{\Delta_0^{\pm}(\lambda;x)}{\Delta(\lambda)},\quad
\int_0^{\infty} I_1^{\pm}(t;x) \,\mathrm{e}^{-\lambda t}\,\mathrm{d}t =
\frac{1}{\sqrt[4]{\lambda}}\frac{\Delta_1^{\pm}(\lambda;x)}{\Delta(\lambda)}.
\]
Concerning the pseudo-distribution of the exit location $X_{\tau_{ab}}$, it is given by
\[
\mathbb{P}_{\!x}\{X_{\tau_{ab}}\in \mathrm{d} z\}/\mathrm{d} z
=H_0^-(x)\,\delta_a(z)-H_1^-(x)\,\delta_a'(z)+H_0^+(x)\,\delta_b(z)-H_1^+(x)\,\delta_b'(z)
\]
with
\begin{align*}
H_0^-(x) &=\frac{(x-b)^2(2x-3a+b)}{(b-a)^3},\quad H_1^-(x)=\frac{(x-a)(x-b)^2}{(b-a)^2},
\\
H_0^+(x) &=-\frac{(x-a)^2(2x+a-3b)}{(b-a)^3},\quad H_1^+(x)=\frac{(x-a)^2(x-b)}{(b-a)^2}.
\end{align*}

When the pseudo-process starts at the middle of the interval $[a,b]$,
we obtain the following expressions for the determinants of interest:
by setting $L=(b-a)/2$,
\begin{align*}
\Delta(\lambda)
&
=32\left[\cosh^2\!\big(\sqrt[4]{\nu}\,L\big)
\sinh^2\!\big(\sqrt[4]{\nu}\,L\big)
-\cos^2\!\big(\sqrt[4]{\nu}\,L\big)
\sin^2\!\big(\sqrt[4]{\nu}\,L\big)\right]\!,
\end{align*}
\begin{align*}
\Delta_0^-\!\left(\lambda;\frac{a+b}{2}\right)=\Delta_0^+\!\left(\lambda;\frac{a+b}{2}\right)
&
=4\left[\cosh\!\big(\sqrt[4]{\nu}\,L\big)\cos\!\big(\sqrt[4]{\nu}\,L\big)
\!\left(\cosh^2\!\big(\sqrt[4]{\nu}\,L\big)+\cos^2\!\big(\sqrt[4]{\nu}\,L\big)-2\right)\right.
\\
&
\left.\hphantom{=\;}
+\sinh\!\big(\sqrt[4]{\nu}\,L\big)\sin\!\big(\sqrt[4]{\nu}\,L\big)
\!\left(\cosh^2\!\big(\sqrt[4]{\nu}\,L\big)-\cos^2\!\big(\sqrt[4]{\nu}\,L\big)\right)\right]\!,
\\
\Delta_1^-\!\left(\lambda;\frac{a+b}{2}\right)=-\Delta_1^+\!\left(\lambda;\frac{a+b}{2}\right)
&
=4\left[\cosh\!\big(\sqrt[4]{\nu}\,L\big)\sin\!\big(\sqrt[4]{\nu}\,L\big)
\sinh^2\!\big(\sqrt[4]{\nu}\,L\big)\right.
\\
&
\left.\hphantom{=\;}
-\sinh\!\big(\sqrt[4]{\nu}\,L\big)\cos\!\big(\sqrt[4]{\nu}\,L\big)
\sin^2\!\big(\sqrt[4]{\nu}\,L\big)\right]\!.
\end{align*}
Hence, in this case, we have the following symmetric expression:
\[
\mathbb{P}_{\!\frac{a+b}{2}}\{\tau_{ab}\in\mathrm{d}t,X_{\tau_{ab}}\in\mathrm{d}z\}
/\mathrm{d}t\,\mathrm{d}z
=I_0^+\!\left(t;\frac{a+b}{2}\right)(\delta_a(z)+\delta_b(z))
+I_1^+\!\left(t;\frac{a+b}{2}\right)(\delta_b'(z)-\delta_a'(z)).
\]
Moreover,
\[
H_0^-\!\left(\lambda;\frac{a+b}{2}\right)=H_0^+\!\left(\lambda;\frac{a+b}{2}\right)=\frac12,\quad
H_1^-\!\left(\lambda;\frac{a+b}{2}\right)=-H_1^+\!\left(\lambda;\frac{a+b}{2}\right)=\frac L4.
\]
Then,
\[
\mathbb{P}_{\!\frac{a+b}{2}}\{X_{\tau_{ab}}\in \mathrm{d} z\}/\mathrm{d} z
=\frac12\,(\delta_a(z)+\delta_b(z))+\frac{b-a}{8}\,(\delta_b'(z)-\delta_a'(z)).
\]

\appendix
\section{Appendix}\label{appendix1}

\subsection{Asymptotics of $\Delta(\lambda)$ and $\Delta_k^{\pm}(\lambda;x)$
as $b$ tends to $+\infty$}

In this appendix, we check limits~(\ref{limits}).
By factorizing the $\ell$th column of the determinant $\Delta(\lambda)$
by $\mathrm{e}_{\lambda}^{\theta_{\ell}\,a}$ for each $\ell\in\{1,\dots,2N\}$
and observing that $\sum_{\ell=1}^{2N}\theta_{\ell}=0$, we find that
\[
\Delta(\lambda)=\begin{vmatrix}
\,1                                                      & \cdots & 1
\\[.5ex]
\,\theta_1                                               & \cdots & \theta_{2N}
\\[.5ex]
\,\vdots                                                 &        &\vdots
\\[.5ex]
\,\theta_1^{N-1}                                         & \cdots & \theta_{2N}^{N-1}
\\[-0.5ex]
\hdotsfor{3}
\\
\,\mathrm{e}_{\lambda}^{\theta_1 (b-a)}                  & \cdots & \mathrm{e}_{\lambda}^{\theta_{2N} (b-a)}
\\[.5ex]
\,\theta_1 \,\mathrm{e}_{\lambda}^{\theta_1 (b-a)}       & \cdots & \theta_{2N} \,\mathrm{e}_{\lambda}^{\theta_{2N} (b-a)}
\\[.5ex]
\,\vdots                                                 &        &\vdots
\\[.5ex]
\,\theta_1^{N-1} \,\mathrm{e}_{\lambda}^{\theta_1 (b-a)} & \cdots & \theta_{2N}^{N-1} \,\mathrm{e}_{\lambda}^{\theta_{2N} (b-a)}
\end{vmatrix}\!.
\]
We separate $\Delta(\lambda)$ into four squared blocks as follows:
\[
\Delta(\lambda)=\begin{vmatrix}
\,V          & \hspace{-1em}\raisebox{-0.5ex}[0ex]{\vdots}\hspace{-1em} & \tilde{V} \\[-1ex]
\hdotsfor{3}\\
\,W(\lambda) & \hspace{-1em}\raisebox{-0.2ex}[0ex]{\vdots}\hspace{-1em} & \tilde{W}(\lambda)
\end{vmatrix}
\]
with
\[
V=\begin{pmatrix}
\,1              & \cdots & 1
\\[.5ex]
\,\theta_1       & \cdots & \theta_N
\\[.5ex]
\,\vdots         &        &\vdots
\\[.5ex]
\,\theta_1^{N-1} & \cdots & \theta_N^{N-1}
\end{pmatrix}\!,\quad
\tilde{V}=\begin{pmatrix}
\,1                  & \cdots & 1
\\[.5ex]
\,\theta_{N+1}       & \cdots & \theta_{2N}
\\[.5ex]
\,\vdots             &        &\vdots
\\[.5ex]
\,\theta_{N+1}^{N-1} & \cdots & \theta_{2N}^{N-1}
\end{pmatrix}\!,
\]
\begin{align*}
W(\lambda)
&
=\begin{pmatrix}
\,\mathrm{e}_{\lambda}^{\theta_1 (b-a)}                  & \cdots & \mathrm{e}_{\lambda}^{\theta_N (b-a)}
\\[.5ex]
\,\theta_1 \,\mathrm{e}_{\lambda}^{\theta_1 (b-a)}       & \cdots & \theta_N \,\mathrm{e}_{\lambda}^{\theta_N (b-a)}
\\[.5ex]
\,\vdots                                                 &        &\vdots
\\[.5ex]
\,\theta_1^{N-1} \,\mathrm{e}_{\lambda}^{\theta_1 (b-a)} & \cdots & \theta_N^{N-1} \mathrm{e}_{\lambda}^{\theta_N (b-a)}
\end{pmatrix}\!,
\\[1ex]
\tilde{W}(\lambda)
&
=\begin{pmatrix}
\,\mathrm{e}_{\lambda}^{\theta_{N+1} (b-a)}                      & \cdots & \mathrm{e}_{\lambda}^{\theta_{2N} (b-a)}
\\[.5ex]
\,\theta_{N+1} \,\mathrm{e}_{\lambda}^{\theta_{N+1} (b-a)}       & \cdots & \theta_{2N} \,\mathrm{e}_{\lambda}^{\theta_{2N} (b-a)}
\\[.5ex]
\,\vdots                                                         &        &\vdots
\\[.5ex]
\,\theta_{N+1}^{N-1} \,\mathrm{e}_{\lambda}^{\theta_{N+1} (b-a)} & \cdots & \theta_{2N}^{N-1} \mathrm{e}_{\lambda}^{\theta_{2N} (b-a)}
\end{pmatrix}\!.
\end{align*}
Due to the fact that $\Re(\theta_{\ell})<0$ for ${\ell}\in\{1,\dots,N\}$ and
$\Re(\theta_{\ell})>0$ for ${\ell}\in\{N+1,\dots,2N\}$, it may be easily seen by using
an expansion by blocks of type $N\times N$ that the leading terms
of $\Delta(\lambda)$ are obtained by performing the product of
the determinants of both diagonal blocks $V$ and $\tilde{W}(\lambda)$,
namely:
\[
\Delta(\lambda)\underset{b\to+\infty}{\sim} \det(V)\times\det(\tilde{W}(\lambda)).
\]

Similarly, we decompose $\Delta_k^-(\lambda;x)$ into
\[
\Delta_k^-(\lambda;x)=\begin{vmatrix}
\,V_k(\lambda;x) & \hspace{-1em}\raisebox{-0.5ex}[0ex]{\vdots}\hspace{-1em} & \tilde{V}_k(\lambda;x) \\[-1ex]
\hdotsfor{3}\\
\,W(\lambda)     & \hspace{-1em}\raisebox{-0.2ex}[0ex]{\vdots}\hspace{-1em} & \tilde{W}(\lambda)
\end{vmatrix}
\]
with
\[
V_k(\lambda;x)=\begin{pmatrix}
\,1                                     & \cdots & 1
\\[.5ex]
\,\theta_1                              & \cdots & \theta_N
\\[.5ex]
\,\vdots                                &        &\vdots
\\[.5ex]
\,\theta_1^{k-1}                        & \cdots & \theta_N^{k-1}
\\[.5ex]
\,\mathrm{e}_{\lambda}^{\theta_1 (x-a)} & \cdots &\,\mathrm{e}_{\lambda}^{\theta_N (x-a)}
\\[.5ex]
\,\theta_1^{k+1}                        & \cdots & \theta_N^{k+1}
\\[.5ex]
\,\vdots                                &        &\vdots
\\[.5ex]
\,\theta_1^{N-1}                        & \cdots & \theta_N^{N-1}
\end{pmatrix}\!,\quad
\tilde{V}_k(\lambda;x)=\begin{pmatrix}
\,1                                         & \cdots & 1
\\[.5ex]
\,\theta_{N+1}                              & \cdots & \theta_{2N}
\\[.5ex]
\,\vdots                                    &        &\vdots
\\[.5ex]
\,\theta_{N+1}^{k-1}                        & \cdots & \theta_{2N}^{k-1}
\\[.5ex]
\,\mathrm{e}_{\lambda}^{\theta_{N+1} (x-a)} & \cdots &\,\mathrm{e}_{\lambda}^{\theta_{2N} (x-a)}
\\[.5ex]
\,\theta_{N+1}^{k+1}                        & \cdots & \theta_{2N}^{k+1}
\\[.5ex]
\,\vdots                                    &        &\vdots
\\[.5ex]
\,\theta_{N+1}^{N-1}                        & \cdots & \theta_{2N}^{N-1}
\end{pmatrix}\!.
\]
We can easily see that, for $x\in(a,b)$,
\[
\Delta_k^-(\lambda;x)\underset{b\to+\infty}{\sim} \det(V_k(\lambda;x))
\times\det(\tilde{W}(\lambda)).
\]
As a byproduct, we get the first limit of~(\ref{limits}):
\[
\frac{\Delta_k^-(\lambda;x)}{\Delta(\lambda)}\underset{b\to+\infty}{\longrightarrow}
\frac{\det(V_k(\lambda;x))}{\det(V)}.
\]
By expanding the determinant of $V_k(\lambda;x)$ with respect to its
$k$th row, we obtain that
\[
\frac{\det(V_k(\lambda;x))}{\det(V)}=
\sum_{\ell=1}^{N} \alpha_{k\ell}\,\mathrm{e}_{\lambda}^{\theta_{\ell}\,(x-a)}
\]
where the coefficients $\alpha_{k\ell}$ are explicitly written in
Remark~\ref{remark-limit}.

Next, concerning the determinant $\Delta_k^+(\lambda;x)$,
by factorizing the $\ell$th column by $\mathrm{e}_{\lambda}^{\theta_{\ell}\,b}$
for each $\ell\in\{1,\dots,2N\}$, using the identity $\sum_{\ell=1}^{2N}\theta_{\ell}=0$
and permuting the $k$th and $(N+k)$th rows for each $k\in\{1,\dots,N\}$,
we get that
\[
\Delta_k^+(\lambda;x)=(-1)^N\begin{vmatrix}
\,1                                                      & \cdots & 1
\\[.5ex]
\,\vdots                                                 &        &\vdots
\\[.5ex]
\,\theta_1^{k-1}                                         & \cdots & \theta_{2N}^{k-1}
\\[.5ex]
\,\mathrm{e}_{\lambda}^{\theta_1 (x-b)}                  & \cdots & \mathrm{e}_{\lambda}^{\theta_{2N} (x-b)}
\\[.5ex]
\,\theta_1^{k+1}                                         & \cdots & \theta_{2N}^{k+1}
\\[.5ex]
\,\vdots                                                 &        &\vdots
\\[.5ex]
\,\theta_1^{N-1}                                         & \cdots & \theta_{2N}^{N-1}
\\[-0.5ex]
\hdotsfor{3}
\\
\,\mathrm{e}_{\lambda}^{\theta_1 (a-b)}                  & \cdots & \mathrm{e}_{\lambda}^{\theta_{2N} (a-b)}
\\[.5ex]
\,\vdots                                                 &        &\vdots
\\[.5ex]
\,\theta_1^{N-1} \,\mathrm{e}_{\lambda}^{\theta_1 (a-b)} & \cdots & \theta_{2N}^{N-1} \,\mathrm{e}_{\lambda}^{\theta_{2N} (a-b)}
\end{vmatrix}\!.
\]
As previously, we decompose $\Delta_k^+(\lambda;x)$ into
\[
\Delta_k^+(\lambda;x)=(-1)^N\begin{vmatrix}
\,Y_k(\lambda;x) & \hspace{-1em}\raisebox{-0.5ex}[0ex]{\vdots}\hspace{-1em} & \tilde{Y}_k(\lambda;x) \\[-1ex]
\hdotsfor{3}\\
\,Z(\lambda)     & \hspace{-1em}\raisebox{-0.2ex}[0ex]{\vdots}\hspace{-1em} & \tilde{Z}(\lambda)
\end{vmatrix}
=\begin{vmatrix}
\,\tilde{Y}_k(\lambda;x) & \hspace{-1em}\raisebox{-0.5ex}[0ex]{\vdots}\hspace{-1em} & Y_k(\lambda;x) \\[-1ex]
\hdotsfor{3}\\
\,\tilde{Z}(\lambda)     & \hspace{-1em}\raisebox{-0.2ex}[0ex]{\vdots}\hspace{-1em} & Z(\lambda)
\end{vmatrix}
\]
with
\[
Y_k(\lambda;x)=\begin{pmatrix}
\,1                                     & \cdots & 1
\\[.5ex]
\,\vdots                                &        &\vdots
\\[.5ex]
\,\theta_1^{k-1}                        & \cdots & \theta_N^{k-1}
\\[.5ex]
\,\mathrm{e}_{\lambda}^{\theta_1 (x-b)} & \cdots & \mathrm{e}_{\lambda}^{\theta_N (x-b)}
\\[.5ex]
\,\theta_1^{k+1}                        & \cdots & \theta_N^{k+1}
\\[.5ex]
\,\vdots                                &        &\vdots
\\[.5ex]
\,\theta_1^{N-1}                        & \cdots & \theta_N^{N-1}
\end{pmatrix}\!,\quad
\tilde{Y}_k(\lambda;x)=\begin{pmatrix}
\,1                                         & \cdots & 1
\\[.5ex]
\,\vdots                                    &        &\vdots
\\[.5ex]
\,\theta_{N+1}^{k-1}                        & \cdots & \theta_{2N}^{k-1}
\\[.5ex]
\,\mathrm{e}_{\lambda}^{\theta_{N+1} (x-b)} & \cdots & \mathrm{e}_{\lambda}^{\theta_{2N} (x-b)}
\\[.5ex]
\,\theta_{N+1}^{k+1}                        & \cdots & \theta_{2N}^{k+1}
\\[.5ex]
\,\vdots                                    &        &\vdots
\\[.5ex]
\,\theta_{N+1}^{N-1}                        & \cdots & \theta_{2N}^{N-1}
\end{pmatrix}\!,
\]
\begin{align*}
Z(\lambda)
&
=\begin{pmatrix}
\,\mathrm{e}_{\lambda}^{\theta_1 (a-b)}                  & \cdots & \mathrm{e}_{\lambda}^{\theta_N (a-b)}
\\[.5ex]
\,\theta_1 \,\mathrm{e}_{\lambda}^{\theta_1 (a-b)}       & \cdots & \theta_N \,\mathrm{e}_{\lambda}^{\theta_N (a-b)}
\\[.5ex]
\,\vdots                                                 &        &\vdots
\\[.5ex]
\,\theta_1^{N-1} \,\mathrm{e}_{\lambda}^{\theta_1 (a-b)} & \cdots & \theta_N^{N-1} \mathrm{e}_{\lambda}^{\theta_N (a-b)}
\end{pmatrix}\!,
\\[1ex]
\tilde{Z}(\lambda)
&
=\begin{pmatrix}
\,\mathrm{e}_{\lambda}^{\theta_{N+1} (a-b)}                      & \cdots & \mathrm{e}_{\lambda}^{\theta_{2N} (a-b)}
\\[.5ex]
\,\theta_{N+1} \,\mathrm{e}_{\lambda}^{\theta_{N+1} (a-b)}       & \cdots & \theta_{2N} \,\mathrm{e}_{\lambda}^{\theta_{2N} (a-b)}
\\[.5ex]
\,\vdots                                                         &        &\vdots
\\[.5ex]
\,\theta_{N+1}^{N-1} \,\mathrm{e}_{\lambda}^{\theta_{N+1} (a-b)} & \cdots & \theta_{2N}^{N-1} \mathrm{e}_{\lambda}^{\theta_{2N} (a-b)}
\end{pmatrix}\!.
\end{align*}
Finally, by remarking that $\theta_{N+\ell}=-\theta_{\ell}$ for any
$\ell\in\{1,\dots,N\}$, we derive that
\[
\Delta_k^+(\lambda;x)=(-1)^k\begin{vmatrix}
\,U_k(\lambda;x) & \hspace{-1em}\raisebox{-0.5ex}[0ex]{\vdots}\hspace{-1em} & \tilde{U}_k(\lambda;x) \\[-1ex]
\hdotsfor{3}\\
\,W(\lambda)     & \hspace{-1em}\raisebox{-0.2ex}[0ex]{\vdots}\hspace{-1em} & \tilde{W}(\lambda)
\end{vmatrix}
\]
where the matrices $U_k(\lambda;x)$ and $\tilde{U}_k(\lambda;x)$
are deduced from $V_k(\lambda;x)$ and $\tilde{V}_k(\lambda;x)$ by changing
$(x-a)$ into $(b-x)$, that is, $U_k(\lambda;x)=V_k(\lambda;a+b-x)$
and $\tilde{U}_k(\lambda;x)=\tilde{V}_k(\lambda;a+b-x)$.
As a byproduct, we derive the identity
\begin{equation}\label{symmetry}
\Delta_k^+(\lambda;x)=(-1)^k\Delta_k^-(\lambda;a+b-x)
\end{equation}
which is evoked in Remark~\ref{remark-symmetry}.
Thanks to an expansion by blocks, we can see that, for $x\in(a,b)$,
\[
\Delta_k^+(\lambda;x)\underset{b\to+\infty}{\sim} (-1)^k\det(V_k(\lambda;a+b-x))
\times\det(\tilde{W}(\lambda))=o[\det(\tilde{W}(\lambda))].
\]
From this, we deduce the second limit of~(\ref{limits}):
\[
\frac{\Delta_k^+(\lambda;x)}{\Delta(\lambda)}\underset{b\to+\infty}{\longrightarrow}0.
\]

\subsection{Asymptotics of $\Delta(\lambda)$ and $\Delta_k^{\pm}(\lambda;x)$
as $\lambda$ tends to $0^+$ or $+\infty$}

The procedure depicted in the previous subparagraph can be carried out
\textit{mutatis mutandis} in the case where $\lambda$ tends to $+\infty$.
This yields the following limiting result:
\begin{equation}\label{limits2}
\frac{\Delta_k^{\pm}(\lambda;x)}{\Delta(\lambda)}\underset{\lambda\to+\infty}{\longrightarrow}0,
\end{equation}
the rate of convergence being exponential.
Then $I_k^{\pm}(t;x)\underset{t\to 0^+}{\longrightarrow}0$.
Below, we examine the case where $\lambda$ tends to $0^+$.

\subsubsection{Asymptotics of $\Delta(\lambda)$ as $\lambda$ tends to $0^+$}

Set $c=\lambda^{1/(2N)}(b-a)$. The number $c$ tends to $0$.
We expand the exponentials lying in $\Delta(\lambda)$ into power series:
for $k\in\{0,\dots,N-1\}$,
\[
\theta_{\ell}^k \,\mathrm{e}_{\lambda}^{\theta_k (b-a)}
=\sum_{i=0}^{\infty} \theta_{\ell}^{i+k}\, \frac{c^i}{i!}
=\sum_{i=k}^{\infty} \theta_{\ell}^i \,\frac{c^{i-k}}{(i-k)!}.
\]
Then,
\[
\Delta(\lambda)=\begin{vmatrix}
\,1                                                           & \cdots & 1
\\[.5ex]
\,\theta_1                                                    & \cdots & \theta_{2N}
\\[.5ex]
\,\vdots                                                      &        &\vdots
\\[.5ex]
\,\theta_1^{N-1}                                              & \cdots & \theta_{2N}^{N-1}
\\[-0.5ex]
\hdotsfor{3}
\\
\,\sum_{i=0}^{\infty} \,\theta_1^i \frac{c^{i}}{i!}             & \cdots & \sum_{i=0}^{\infty} \,\theta_{2N}^i \frac{c^{i}}{i!}
\\[.5ex]
\,\sum_{i=1}^{\infty} \,\theta_1^i \frac{c^{i-1}}{(i-1)!}       & \cdots & \sum_{i=1}^{\infty} \,\theta_{2N}^i \frac{c^{i-1}}{(i-1)!}
\\[.5ex]
\,\vdots                                                        &        & \vdots
\\[.5ex]
\,\sum_{i=N-1}^{\infty} \,\theta_1^i \frac{c^{i-N+1}}{(i-N+1)!} & \cdots & \sum_{i=N-1}^{\infty} \,\theta_{2N}^i \frac{c^{i-N+1}}{(i-N+1)!}
\end{vmatrix}\!.
\]
By multilinearity, we see that the terms including a power of $\theta_{\ell}$
less than $N$ can be discarded (for these terms, the corresponding determinant
has two or more identical rows, thus it vanishes). Hence, the determinant $\Delta(\lambda)$
does not change if we only keep the sums
$\sum_{i=N}^{\infty} \theta_{\ell}^i \,c^{i-k}/(i-k)!$:
\[
\Delta(\lambda)=\begin{vmatrix}
\,1                                                           & \cdots & 1
\\[.5ex]
\,\theta_1                                                    & \cdots & \theta_{2N}
\\[.5ex]
\,\vdots                                                      &        &\vdots
\\[.5ex]
\,\theta_1^{N-1}                                              & \cdots & \theta_{2N}^{N-1}
\\[-0.5ex]
\hdotsfor{3}
\\
\,\sum_{i=N}^{\infty} \,\theta_1^i \frac{c^{i}}{i!}           & \cdots & \sum_{i=N}^{\infty} \,\theta_{2N}^i \frac{c^{i}}{i!}
\\[.5ex]
\,\sum_{i=N}^{\infty} \,\theta_1^i \frac{c^{i-1}}{(i-1)!}     & \cdots & \sum_{i=N}^{\infty} \,\theta_{2N}^i \frac{c^{i-1}}{(i-1)!}
\\[.5ex]
\,\vdots                                                      &        & \vdots
\\[.5ex]
\,\sum_{i=N}^{\infty} \,\theta_1^i \frac{c^{i-N+1}}{(i-N+1)!} & \cdots & \sum_{i=N}^{\infty} \,\theta_{2N}^i \frac{c^{i-N+1}}{(i-N+1)!}
\end{vmatrix}\!.
\]
By multilinearity, we can rewrite $\Delta(\lambda)$ as
\[
\Delta(\lambda)=\sum_{i_1,\dots,i_N \ge N\atop i_1,\dots,i_N\text{ all distinct}}
\frac{c^{i_1+(i_2-1)+\dots+(i_N-N+1)}}{i_1!(i_2-1)!\cdots(i_N-N+1)!}
\begin{vmatrix}
\,1              & \cdots & 1
\\[.5ex]
\,\theta_1       & \cdots & \theta_{2N}
\\[.5ex]
\,\vdots         &        & \vdots
\\[.5ex]
\,\theta_1^{N-1} & \cdots & \theta_{2N}^{N-1}
\\[-0.5ex]
\hdotsfor{3}
\\
\,\theta_1^{i_1} & \cdots & \theta_{2N}^{i_1}
\\[.5ex]
\,\theta_1^{i_2} & \cdots & \theta_{2N}^{i_2}
\\[.5ex]
\,\vdots         &        & \vdots
\\[.5ex]
\,\theta_1^{i_N} & \cdots & \theta_{2N}^{i_N}
\end{vmatrix}\!.
\]
Because of the conditions on the indices $i_1,\dots,i_N$, the least power
of $c$ is not less than $N^2$: indeed,
the indices being distinct and not less than $N$, we have
$i_1+i_2+\dots+i_N\ge N+(N+1)+\dots+(2N-1)$
or, equivalently, $i_1+(i_2-1)+\dots+(i_N-N+1)\ge N^2$.
Moreover, if an index is greater than $(2N-1)$, say $i_N\ge 2N$,
then $i_1+i_2+\dots+i_{N-1}+i_N\ge N+(N+1)+\dots+(2N-2)+(2N)$,
that is, $i_1+(i_2-1)+\dots+(i_N-N+1)\ge N^2+1$.
In words, the term $c^{N^2}$ is obtained at most for the
indices not greater than $(2N-1)$. Consequently, we see that
the terms of the sums corresponding to $i$ greater
than $(2N-1)$ can be neglected when $c$ tends to 0, namely:
\[
\Delta(\lambda) \underset{c\to 0^+}{=}\begin{vmatrix}
\,1                                                         & \cdots & 1
\\[.5ex]
\,\theta_1                                                  & \cdots & \theta_{2N}
\\[.5ex]
\,\vdots                                                    &        &\vdots
\\[.5ex]
\,\theta_1^{N-1}                                            & \cdots & \theta_{2N}^{N-1}
\\[-0.5ex]
\hdotsfor{3}
\\
\,\sum_{i=N}^{2N-1} \,\theta_1^i \frac{c^{i}}{i!}           & \cdots & \sum_{i=N}^{2N-1} \,\theta_{2N}^i \frac{c^{i}}{i!}
\\[.5ex]
\,\sum_{i=N}^{2N-1} \,\theta_1^i \frac{c^{i-1}}{(i-1)!}     & \cdots & \sum_{i=N}^{2N-1} \,\theta_{2N}^i \frac{c^{i-1}}{(i-1)!}
\\[.5ex]
\,\vdots                                                    &        & \vdots
\\[.5ex]
\,\sum_{i=N}^{2N-1} \,\theta_1^i \frac{c^{i-N+1}}{(i-N+1)!} & \cdots & \sum_{i=N}^{2N-1} \,\theta_{2N}^i \frac{c^{i-N+1}}{(i-N+1)!}
\end{vmatrix}+o\!\left(c^{N^2}\right)
\]
We observe that the matrix lying in the foregoing determinant
can be factorized into the product of the two following matrices:
\[
A_1=\begin{pmatrix}
\,I & \hspace{-1em}\raisebox{-0.5ex}[0ex]{\vdots}\hspace{-1em} & O \\[-1ex]
\hdotsfor{3}\\
\,O & \hspace{-1em}\raisebox{-0.2ex}[0ex]{\vdots}\hspace{-1em} & B
\end{pmatrix}\!,
\quad
A_2=\begin{pmatrix}
\,1               & \cdots & 1
\\[.5ex]
\,\theta_1        & \cdots & \theta_{2N}
\\[.5ex]
\,\vdots          &        &\vdots
\\[.5ex]
\,\theta_1^{2N-1} & \cdots & \theta_{2N}^{2N-1}
\end{pmatrix}
\]
where $I$ and $O$ are respectively the unit and zero matrices of type $N\times N$,
and
\[
B=\begin{pmatrix}
\,\frac{c^{N}}{N!}       & \frac{c^{N+1}}{(N+1)!} & \cdots & \frac{c^{2N-1}}{(2N-1)!}
\\[.5ex]
\,\frac{c^{N-1}}{(N-1)!} & \frac{c^{N}}{N!}       & \cdots & \frac{c^{2N-2}}{(2N-2)!}
\\[.5ex]
\,\vdots                 & \vdots                 &        & \vdots
\\[.5ex]
\,c                      & \frac{c^{2}}{2!}       & \cdots & \frac{c^{N}}{N!}
\end{pmatrix}\!.
\]
We can decompose $B$ into $C_1\tilde{B}C_2$ where $C_1$ and $C_2$ are the diagonal matrices
with $c^N,c^{N-1},\dots,c$ and $1,c,\dots,c^{N-1}$ as diagonal terms respectively,
and
\[
\tilde{B}=\begin{pmatrix}
\,\frac{1}{N!}     & \frac{1}{(N+1)!} & \cdots & \frac{1}{(2N-1)!}
\\[.5ex]
\,\frac{1}{(N-1)!} & \frac{1}{N!}     & \cdots & \frac{1}{(2N-2)!}
\\[.5ex]
\,\vdots           & \vdots           &        & \vdots
\\[.5ex]
\,1                & \frac{1}{2!}     & \cdots & \frac{1}{N!}
\end{pmatrix}\!.
\]
Hence, all this discussion plainly entails that
\[
\Delta(\lambda) \underset{c \to 0^+}{\sim} \det(A_1)\times\det(A_2)
=\det(A_2)\times\det(\tilde{B})\times\det(C_1)\times\det(C_2)
=constant \times c^{N^2}
\]
where the constant does not vanish, or, by means of the variable $\lambda$,
\begin{equation}\label{asymptotics1}
\Delta(\lambda) \underset{\lambda\to 0^+}{\sim} constant \times
\lambda^{N/2}.
\end{equation}

\subsubsection{Asymptotics of $\Delta_k^{\pm}(\lambda;x)$ as $\lambda$ tends to $0^+$}

A similar analysis can be carried out in the case of the determinant
$\Delta_k^{\pm}(\lambda;x)$. Recall that $c=\lambda^{1/(2N)}(b-a)$
and set $\gamma=\lambda^{1/(2N)}$ $(x-a)$.
The numbers $c$ and $\gamma$ tend to $0$ as $\lambda$ tends to $0^+$. E.g., for
$\Delta_k^-(\lambda;x)$, we have that
\[
\Delta_k^{\pm}(\lambda;x)=\begin{vmatrix}
\,1                                                             & \cdots & 1
\\[.5ex]
\,\theta_1                                                      & \cdots & \theta_{2N}
\\[.5ex]
\,\vdots                                                        &        &\vdots
\\[.5ex]
\,\theta_1^{k-1}                                                & \cdots & \theta_{2N}^{k-1}
\\
\,\sum_{i=0}^{\infty} \,\theta_1^i \frac{\gamma^i}{i!}          & \cdots & \sum_{i=0}^{\infty} \,\theta_{2N}^i \frac{\gamma^i}{i!}
\\[1ex]
\,\theta_1^{k+1}                                                & \cdots & \theta_{2N}^{k+1}
\\[.5ex]
\,\vdots                                                        &        &\vdots
\\[.5ex]
\,\theta_1^{N-1}                                                & \cdots & \theta_{2N}^{N-1}
\\[-0.5ex]
\hdotsfor{3}
\\
\,\sum_{i=0}^{\infty} \,\theta_1^i \frac{c^{i}}{i!}             & \cdots & \sum_{i=0}^{\infty} \,\theta_{2N}^i \frac{c^{i}}{i!}
\\[.5ex]
\,\sum_{i=1}^{\infty} \,\theta_1^i \frac{c^{i-1}}{(i-1)!}       & \cdots & \sum_{i=1}^{\infty} \,\theta_{2N}^i \frac{c^{i-1}}{(i-1)!}
\\[.5ex]
\,\vdots                                                        &        & \vdots
\\[.5ex]
\,\sum_{i=N-1}^{\infty} \,\theta_1^i \frac{c^{i-N+1}}{(i-N+1)!} & \cdots & \sum_{i=N-1}^{\infty} \,\theta_{2N}^i \frac{c^{i-N+1}}{(i-N+1)!}
\end{vmatrix}\!.
\]
As previously, this determinant remains unchanged by removing the terms
related to the indices $0,1,\dots,k-1,k+1,\dots,N-1$ in each sum.
Moreover, for obtaining an asymptotics when $c,\gamma$ tend to 0
(actually $c$ and $\gamma$ have the same order of growth when $\lambda$ tends to 0),
it is enough to keep the terms related to the indices not greater than $(2N-1)$.
Then, by setting $I_k=\{k\}\cup \{N,N+1,\dots,2N-1\}$,
\[
\Delta_k^-(\lambda;x)\underset{c, \gamma\to 0^+}{=}\begin{vmatrix}
\,1                                                                     & \cdots & 1
\\[.5ex]
\,\theta_1                                                              & \cdots & \theta_{2N}
\\[.5ex]
\,\vdots                                                                &        &\vdots
\\[.5ex]
\,\theta_1^{k-1}                                                        & \cdots & \theta_{2N}^{k-1}
\\
\,\sum_{i\in I_k} \,\theta_1^i \frac{\gamma^i}{i!}                      & \cdots & \sum_{i\in I_k} \,\theta_{2N}^i \frac{\gamma^i}{i!}
\\[1ex]
\,\theta_1^{k+1}                                                        & \cdots & \theta_{2N}^{k+1}
\\[.5ex]
\,\vdots                                                                &        &\vdots
\\[.5ex]
\,\theta_1^{N-1}                                                        & \cdots & \theta_{2N}^{N-1}
\\[-0.5ex]
\hdotsfor{3}
\\
\,\sum_{i\in I_k} \,\theta_1^i \frac{c^{i}}{i!}                         & \cdots & \sum_{i\in I_k} \,\theta_{2N}^i \frac{c^{i}}{i!}
\\[.5ex]
\,\sum_{i\in I_k\atop i\ge 1} \,\theta_1^i \frac{c^{i-1}}{(i-1)!}       & \cdots & \sum_{i\in I_k\atop i\ge 1} \,\theta_{2N}^i \frac{c^{i-1}}{(i-1)!}
\\[.5ex]
\,\vdots                                                                &        & \vdots
\\[.5ex]
\,\sum_{i\in I_k\atop i\ge N-1} \,\theta_1^i \frac{c^{i-N+1}}{(i-N+1)!} & \cdots & \sum_{i\in I_k\atop i\ge N-1} \,\theta_{2N}^i \frac{c^{i-N+1}}{(i-N+1)!}
\end{vmatrix}+o\!\left(\gamma^k c^{N^2}\right)\!.
\]
We observe that the matrix lying in the above determinant is the product
of $\tilde{A}_1$ by $A_2$ where
\begin{itemize}
\item
$\tilde{A}_1=\begin{pmatrix}
\,I_k     & \hspace{-1em}\raisebox{-0.5ex}[0ex]{\vdots}\hspace{-1em} & O_{2,k} \\[-1ex]
\hdotsfor{3}\\
\,O_{1,k} & \hspace{-1em}\raisebox{-0.2ex}[0ex]{\vdots}\hspace{-1em} & B
\end{pmatrix}$;
\item
$I_k$ is the diagonal matrix of type $N\times N$ with diagonal terms
equal to $1$ except for the $(k+1)$th which is $\gamma^k/k!$;
\item
$O_{1,k}$ is the matrix of type $N\times N$ with all terms equal to 0
except for the $(k+1)$th column which is made of
$c^k/k!,c^{k-1}/(k-1)!,\dots,c,1,0,\dots,0$;
\item
$O_{2,k}$ is the matrix of type $N\times N$ with all terms equal to 0
except for the $(k+1)$th row which is made of
$\gamma^N/N!,$ $\gamma^{N+1}/(N+1)!,\dots,\gamma^{2N-1}/(2N-1)!$.
\end{itemize}
The determinant of $\tilde{A}_1$ remains unchanged by interchanging
its $(k+1)$th and $N$th columns and its $(k+1)$th and $(k+1)$th rows.
This yields that
\begin{align*}
\det(\tilde{A}_1)
&
=\left|\begin{array}{@{\hspace{0.1em}}cccc;{.08em/.3em}ccc@{\hspace{0.2em}}}
\,1      & \cdots & 0       & 0                      & 0                         & \cdots & 0
\\[.5ex]
\,\vdots & \ddots & \vdots  & \vdots                 & \vdots                    &        & \vdots
\\[.5ex]
\,0      & \cdots & 1       & 0                      & 0                         & \cdots & 0
\\
\,0      & \cdots & 0       & \frac{\gamma^k}{k!}    & \frac{\gamma^N}{N!}       & \cdots & \frac{\gamma^{2N-1}}{(2N-1)!}
\\[-1ex]
\hdotsfor{7}
\\
\,0      & \cdots & 0       & \frac{c^k}{k!}         & \frac{c^N}{N!}             & \cdots & \frac{c^{2N-1}}{(2N-1)!}
\\[.5ex]
\,0      & \cdots & 0       & \frac{c^{k-1}}{(k-1)!} & \frac{c^{N-1}}{(N-1)!}     & \cdots & \frac{c^{2N-2}}{(2N-2)!}
\\[.5ex]
\,\vdots &        & \vdots  & \vdots                 & \vdots                     &        & \vdots
\\[.5ex]
\,0      & \cdots & 0       & 1                      & \frac{c^{N-k}}{(N-k)!}     & \cdots & \frac{c^{2N-k-1}}{(2N-k-1)!}
\\[.5ex]
\,0      & \cdots & 0       & 0                      & \frac{c^{N-k-1}}{(N-k-1)!} & \cdots & \frac{c^{2N-k-2}}{(2N-k-2)!}
\\[.5ex]
\,\vdots &        & \vdots  & \vdots                 & \vdots                     &        & \vdots
\\[.5ex]
\,0      & \cdots & 0       & 0                      & c                          & \cdots & \frac{c^N}{N!}
\end{array}\right|
\\[2ex]
&
=\begin{vmatrix}
\,\frac{\gamma^k}{k!}    & \frac{\gamma^N}{N!}       & \cdots & \frac{\gamma^{2N-1}}{(2N-1)!}
\\[.5ex]
\,\frac{c^k}{k!}         & \frac{c^N}{N!}             & \cdots & \frac{c^{2N-1}}{(2N-1)!}
\\[.5ex]
\,\frac{c^{k-1}}{(k-1)!} & \frac{c^{N-1}}{(N-1)!}     & \cdots & \frac{c^{2N-2}}{(2N-2)!}
\\[.5ex]
\,\vdots                 & \vdots                     &        & \vdots
\\[.5ex]
\,1                      & \frac{c^{N-k}}{(N-k)!}     & \cdots & \frac{c^{2N-k-1}}{(2N-k-1)!}
\\[.5ex]
\,0                      & \frac{c^{N-k-1}}{(N-k-1)!} & \cdots & \frac{c^{2N-k-2}}{(2N-k-2)!}
\\[.5ex]
\,\vdots                 & \vdots                     &        & \vdots
\\[.5ex]
\,0                      & c                          & \cdots & \frac{c^N}{N!}
\end{vmatrix}\!.
\end{align*}
By expanding this last determinant with respect to its first row,
it is not difficult to see that $\det(\tilde{A}_1)=O\left(\gamma^k c^{N^2}\right)$
(recall that $c$ and $\gamma$ have the same order of growth when $\lambda$ tends to 0).
Therefore, in terms of the variable $\lambda$,
\begin{equation}\label{asymptotics2}
\Delta_k^-(\lambda;x)\underset{\lambda\to 0^+}{=}O\big(\lambda^{k/(2N)+N/2}\big)
\end{equation}
and the same holds for $\Delta_k^+(\lambda;x)$. Finally, by~(\ref{asymptotics1})
and~(\ref{asymptotics2}), we derive that
\begin{equation}\label{asymptotics}
\frac{\Delta_k^-(\lambda;x)}{\Delta(\lambda)}
\underset{\lambda\to 0^+}{=}O\big(\lambda^{k/(2N)}\big).
\end{equation}



\begin{thebibliography}{99}

\bibitem{asf} Albeverio, S., Smorodina, N. and Faddeev, M.:
The probabilistic representation of the exponent of a class of
pseudo-differential operators.
\emph{J. Global and Stochastic Analysis} \textbf{1} no.~2, (2011), 123--148.

\bibitem{bho} Beghin, L., Hochberg, K. J. and Orsingher, E.:
Conditional maximal distributions of processes related to higher-order heat-type equations.
\emph{Stochastic Process. Appl.} \textbf{85} no.~2, (2000), 209--223. 

\bibitem{bo} Beghin, L. and Orsingher, E.:
The distribution of the local time for ``pseudoprocesses'' and its connection
with fractional diffusion equations.
\emph{Stochastic Process. Appl.} \textbf{115} no.~6, (2005), 1017--1040. 

\bibitem{bor} Beghin, L., Orsingher, O. and Ragozina, T.:
Joint distributions of the maximum and the process for higher-order diffusions.
\emph{Stochastic Process. Appl.} \textbf{94} no.~1, (2001), 71--93. 

\bibitem{cl1} Cammarota, V. and Lachal, A.:
Joint distribution of the process and its sojourn time in the positive half-line
for pseudo-processes governed by high-order heat equation.
\emph{Electron. J. Probab.} \textbf{15}, (2010), 895--931. 

\bibitem{cl2} Cammarota, V. and Lachal, A.:
Joint distribution of the process and its sojourn time in a half-line for
pseudo-processes governed by higher-order heat-type equations.
\emph{Stochastic Process. Appl.} \textbf{122} no.~1, (2012), 217--249. 

\bibitem{df} Daletskii, Yu. L. and Fomin, S. V.:
Generalized Measures in Function Spaces.
\emph{Theory Probab. Appl.} \textbf{10} no.~2, (1965), 304--316.

\bibitem{hoch} Hochberg, K. J.:
A signed measure on path space related to Wiener measure.
\emph{Ann. Probab.} \textbf{6} no.~3, (1978), 433--458. 

\bibitem{ho} Hochberg, K. J. and Orsingher, E.:
The arc-sine law and its analogs for processes governed by signed and complex measures.
\emph{Stochastic Process. Appl.} \textbf{52} no.~2, (1994), 273--292. 

\bibitem{isf} Ibragimov, I. A., Smorodina, N. V. and Faddeev., M. M.:
Probabilistic approach to the construction of one-dimensional
initial-boundary value solution.
\emph{Teoriya Veroyatnostei i Primenen} \textbf{58} no.~1, (2013).

\bibitem{kry} Krylov, V. Yu.:
Some properties of the distribution corresponding to the equation
$\frac{\partial u}{\partial t}=(-1)^{q+1}\frac{\partial^{2q} u}{\partial^{2q} x}$.
\emph{Soviet Math. Dokl.} \textbf{1}, (1960), 760--763. 

\bibitem{la1} Lachal, A.:
Distribution of sojourn time, maximum and minimum for
pseudo-processes governed by higer-order heat-typer equations.
\emph{Electron. J. Probab.} \textbf{8}, paper no. 20, (2003), 1--53. 

\bibitem{la2} Lachal, A.:
Joint law of the process and its maximum, first hitting time and place of
a half-line  for the pseudo-process driven by the equation
$\frac{\partial}{\partial t}=\pm\frac{\partial^N}{\partial x^N}$.
\emph{C. R. Math. Acad. Sci. Paris} \textbf{343} no.~8, (2006), 525--530. 

\bibitem{la3} Lachal, A.:
First hitting time and place, monopoles and multipoles for pseudo-processes
driven by the equation $\frac{\partial}{\partial t}=\pm \frac{\partial^N}{\partial x^N}$.
\emph{Electron. J. Probab.} \textbf{12} no.~29, (2007), 300--353. 

\bibitem{la4} Lachal, A.:
First hitting time and place for the pseudo-processes driven
by the equation $\frac{\partial}{\partial t}=\pm\frac{\partial^n}{\partial x^n}$
subject to a linear drift.
\emph{Stochastic Process. Appl.} \textbf{118} no.~1, (2008), 1--27. 

\bibitem{la5}
Lachal, A.: A survey on the pseudo-process driven by the high-order heat-type
equation $\partial/\partial t=\pm\partial^N/\partial x^N$
concerning the first hitting times and sojourn times.
\emph{Methodology and Computing in Applied Probability} \textbf{14} no.~3, (2012), 549--566.

\bibitem{la6}
Lachal, A.: From pseudo-random walk to pseudo-Brownian motion:
first exit time from a one-sided or a two-sided interval.
\emph{International Journal of Stochastic Analysis (In press).}

\bibitem{ns} Nakajima, T. and Sato, S.:
On the joint distribution of the first hitting time and the first hitting place
to the space-time wedge domain of a biharmonic pseudo process.
\emph{Tokyo J. Math.} \textbf{22} no.~2, (1999), 399--413. 

\bibitem{no}
Nikitin, Ya. Yu. and Orsingher, E. (2000): On sojourn distributions of
processes related to some higher-order heat-type equations.
\emph{J. Theoret. Probab.} \textbf{13} no.~4, 997--1012. 

\bibitem{nish1} Nishioka, K.:
Monopoles and dipoles of a biharmonic pseudo process.
\emph{Proc. Japan Acad. Ser. A} \textbf{72}, (1996), 47--50. 

\bibitem{nish2} Nishioka, K.:
The first hitting time and place of a half-line by a biharmonic pseudo process.
\emph{Japan J. Math.} \textbf{23}, (1997), 235--280. 

\bibitem{nish3} Nishioka, K.:
Boundary conditions for one-dimensional biharmonic pseudo process.
\emph{Electron. J. Probab.} \textbf{6}, paper no.~13, (2001), 1--27. 

\bibitem{ors} Orsingher, E.:
Processes governed by signed measures connected with third-order ``heat-type'' equations.
\emph{Lithuanian Math. J.} \textbf{31} no.~2, (1991), 220--231. 

\end{thebibliography}
\end{document}